\documentclass[11pt]{article}

\usepackage{geometry} % document margins
\usepackage[T1]{fontenc} % font encoding
\usepackage{lmodern} % default english font % times

\usepackage{titlesec}
\titleformat*{\section}{\Large\bfseries} % section font (jlreq)
\titleformat*{\subsection}{\large\bfseries}
\titleformat*{\subsubsection}{\bfseries}

\usepackage{enumitem} % extending list environments % [label=\arabic*]
\setlist[itemize]{labelsep=1pt, partopsep=-3pt} % itemize

\setlist[enumerate]{labelsep=4pt, partopsep=-3pt} % enumerate

\setlist[description]{labelsep=6pt, partopsep=-6pt} % description

\usepackage{comment} % comment out multiple lines
\usepackage{array} % extending \array and \tabular
\usepackage{listings} % typeset source codes % \begin{lstlisting}
\usepackage[normalem]{ulem} % underlines % break lines only english
\usepackage{ascmac} % screen, itembox, shadebox
\usepackage{tcolorbox} % colored box % including graphicx, xcolor, tikz
\usepackage{subcaption}
\tcbuselibrary{breakable, skins, theorems} % tcolorbox library

\newcounter{rnum} % Don't use otf package if you use pdflatex

\usepackage{amsmath}
\usepackage{amssymb}
\usepackage{amsthm}
\usepackage{mathtools}
\usepackage{bm}

\numberwithin{equation}{section} % section number in equation number
\allowdisplaybreaks

\usepackage{hyperref}
\hypersetup{
    setpagesize=false, %bookmarksdepth=tocdepth, bookmarksnumbered=true,
    colorlinks=true, citecolor=blue, %green!75!black,
    linkcolor=blue, urlcolor=cyan!65!black,
}
\usepackage{cleveref}
\expandafter\def\csname ver@etex.sty\endcsname{3000/12/31} % get rid of the warning about etex

\usepackage{autonum}

\theoremstyle{plain} % modify \begin{theorem}[title]\label{}

\newtheorem{theorem}{Theorem}[section]
\newtheorem{proposition}[theorem]{Proposition}
\newtheorem{lemma}[theorem]{Lemma}
\newtheorem{corollary}[theorem]{Corollary}

\newtheorem{theorem*}{Theorem}[section]
\newtheorem{proposition*}[theorem]{Proposition}
\newtheorem{lemma*}[theorem]{Lemma}
\newtheorem{corollary*}[theorem]{Corollary}
\newtheorem{definition*}[theorem]{Definition}
\newtheorem{assumption*}[theorem]{Assumption}

\theoremstyle{remark}

\newtheorem{remark*}[theorem]{Remark}
\newtheorem{example*}[theorem]{Example}

\crefname{theorem}{Theorem}{Theorems}
\crefname{proposition}{Proposition}{Propositions}
\crefname{lemma}{Lemma}{Lemmas}
\crefname{corollary}{Corollary}{Corollaries}
\crefname{definition}{Definition}{Definitions}
\crefname{assumption}{Assumption}{Assumptions}
\crefname{remark}{remark}{remarks}
\crefname{example}{example}{examples}

\crefname{equation}{Equation}{Equations}
\crefname{figure}{Figure}{Figures}
\crefname{table}{Table}{Tables}
\crefname{algorithm}{Algorithm}{Algorithm}

\crefname{section}{Section}{Sections}
\crefname{subsection}{Subsection}{Subsections}
\crefname{appendix}{Appendix}{Appendices}
\Crefname{appendix}{Appendix}{Appendices}

\DeclarePairedDelimiterXPP\Prob[2]{#1}{\lparen}{\rparen}{}{#2}
\DeclarePairedDelimiterXPP\Exp[2]{\mathbb{E}_{#1}}{\lbrack}{\rbrack}{}{#2}

\cslet{blx@noerroretextools}\empty % to avoid incompatibility with autonum (etextool)
\usepackage[backend=biber,style=alphabetic]{biblatex}
\DeclareDelimFormat[bib]{nametitledelim}{\addcolon\space}

\renewbibmacro{in:}{}

\renewbibmacro*{volume+number+eid}{
    \printfield{volume}
    \setunit*{\addnbspace}
    \printfield{number}
    \setunit{\addcomma\space}
    \printfield{eid}
}
\DeclareFieldFormat[article]{volume}{\mkbibbold{#1}}
\DeclareFieldFormat[article]{number}{\mkbibparens{#1}}

\renewbibmacro*{issue+date}{}
\renewbibmacro*{note+pages}{
    \printfield{note}
    \setunit{\bibpagespunct}
    \printfield{pages}
    \setunit{\addcomma\addspace}
    \usebibmacro{date}
    \newunit
}
\DeclareFieldFormat{pages}{#1}
\DeclareFieldFormat{date}{\mkbibparens{#1}}

\DeclareFieldFormat{parenswithperiod}{\mkbibbrackets{#1\addperiod}}

\renewbibmacro*{pageref}{
    \iflistundef{pageref}{}{
        \printtext[parenswithperiod]{
            \ifnumgreater{\value{pageref}}{1}
                {\bibstring{backrefpages}\ppspace}
                {\bibstring{backrefpage}\ppspace}
            \printlist[pageref][-\value{listtotal}]{pageref}
        }
        
    }
}
\DefineBibliographyStrings{english}{
    backrefpage  = {\hspace{-0.8mm}cit.\,on page\hspace{-1.2mm}},
    backrefpages = {\hspace{-0.8mm}cit.\, on pages\hspace{-1.2mm}},
}
\DeclareCiteCommand{\citeauthor}{
    \boolfalse{citetracker}%
    \boolfalse{pagetracker}%
    \usebibmacro{prenote}
    }{\ifciteindex
        {\indexnames{labelname}}
        {}%
    \printtext[bibhyperref]{\printnames{labelname}\!}
    }{\multicitedelim}{\usebibmacro{postnote}
}
\DeclareCiteCommand{\citeauthoryear}{
    \boolfalse{citetracker}%
    \boolfalse{pagetracker}%
    \usebibmacro{prenote}
    }{\ifciteindex
        {\indexnames{labelname}}
        {}%
    \printtext[bibhyperref]{\printnames{labelname}}\printtext{\hspace{1mm}(\printfield{year})\hspace{-2mm}}
    }{\multicitedelim}{\usebibmacro{postnote}
}

\begin{document}

\title{Error Distribution of the Local Linearization Method for Stochastic Differential Equations with Additive Brownian Noise}

\author{
Masaaki Fukasawa\footnotemark[1]
\and
Mikio Hirokane\footnotemark[1]\,\,$^,$\,\footnotemark[2]
\and
Kostas Kardaras\footnotemark[3]
}

\date{}

\maketitle

\begingroup
\renewcommand{\thefootnote}{\fnsymbol{footnote}}
\footnotetext[1]{Graduate School of Engineering Science, The University of Osaka.}
\footnotetext[2]{E-mail: \texttt{hirokane@sigmath.es.osaka-u.ac.jp}.}
\footnotetext[3]{Department of Statistics, London School of Economics and Political Science.}
\endgroup

\begin{abstract}
We prove a functional stable limit theorem for the discretization error process
of a local linearization scheme for stochastic differential equations with
additive Brownian noise. The scheme includes the conditional mean of the
second-order term involving the Brownian increment in the Taylor expansion of
the drift. The leading error is then formed by centered quadratic terms in the
Brownian increments, and the sharp normalization is \(n\sqrt n\). Under
\(C^3\)-regularity and a Lyapunov-type condition on the drift, the scaled error
process converges stably in \(C([0,1],\mathbb R^d)\) to the solution of the
limiting linear stochastic differential equation. The martingale part of the
limit is driven by a Brownian motion independent of the original
\(\sigma\)-field, and its coefficient is determined by the Hessian of the
drift and the covariance matrix of the additive noise.

\end{abstract}

% \tableofcontents

\section{Introduction}
\label{sec:introduction}

Discretization schemes for stochastic differential equations (SDEs), and
limit theorems for their scaled error processes, have been extensively studied
since the weak convergence framework of Kurtz and Protter
\cite{kurtz1991weak} and the asymptotic error theorem for the Euler-Maruyama scheme due to Jacod and Protter
\cite{jacod1998asymptotic}.
On \([0,1]\), consider the \(d\)-dimensional SDE with multiplicative Brownian noise
\begin{equation}\label{eq:intro_em_reference_sde}
    dY_t
    =
    f(Y_t)\,dt
    +
    \sigma(Y_t)\,dB_t,
    \qquad
    Y_0=y_0,
\end{equation}
where \(B\) is a \(q\)-dimensional Brownian motion,
\(f:\mathbb R^d\to\mathbb R^d\), and \(\sigma:\mathbb R^d\to\mathbb R^{d\times q}\).
For \(\varphi_n(t)\coloneqq \lfloor nt\rfloor/n\), its continuous
Euler--Maruyama approximation is
\begin{equation}
    d\hat{Y}_t^{n}
    =
    f(\hat{Y}_{\varphi_n(t)}^{n})\,dt
    +
    \sigma(\hat{Y}_{\varphi_n(t)}^{n})\,dB_t,
    \qquad
    \hat{Y}_0^{n}=y_0 .
\end{equation}
If \(f\) and \(\sigma\) are \(C^1\) and satisfy the linear growth condition, applying the result of \cite{jacod1998asymptotic} to this case yields the stable convergence of
\[
    \hat{V}_t^{n}
    \coloneqq
    \sqrt n\bigl(Y_t-\hat{Y}_t^{n}\bigr)
\]
in \(C([0,1],\mathbb R^d)\) to the solution \(V=(V^1,V^2,\dots, V^d)\) of
\begin{equation}\label{eq:intro_jp_euler_limit}
\begin{aligned}
    dV_t^{i}
    ={}&
    \sum_{k=1}^d
    \partial_k f^i(Y_t)V_t^{k}\,dt
    +
    \sum_{\alpha=1}^q\sum_{k=1}^d
    \partial_k\sigma^i_\alpha(Y_t)V_t^{k}\,dB_t^\alpha
    \\
    &\quad
    +
    \frac1{\sqrt2}
    \sum_{\alpha=1}^q\sum_{\beta=1}^q\sum_{k=1}^d
    \partial_k\sigma^i_\alpha(Y_t)\sigma^k_\beta(Y_t)
    \,d\overline B_t^{\beta\alpha},
    \qquad
    V_0^{i}=0,
\end{aligned}
\end{equation}
for \(i=1,\ldots,d\), where
\(\overline B=(\overline B^{\beta\alpha})_{1\le \beta,\alpha\le q}\) is a
\(q^2\)-dimensional Brownian motion independent of the original \(\sigma\)-field.
Thus, for SDEs with multiplicative Brownian noise, the
Euler--Maruyama error has the \(n^{1/2}\) normalization, and its limit
involves both the original Brownian motion and an additional independent
Brownian motion.  If \(\sigma\) is constant, then all terms involving
derivatives of \(\sigma\) vanish.  In this case,
\eqref{eq:intro_jp_euler_limit} reduces to a homogeneous linear equation
with zero initial condition, and hence \(V\equiv0\).  The
\(n^{1/2}\)-normalization therefore gives a degenerate limit in the
additive noise case.  Extensions of the Euler error limit under weaker
assumptions on the coefficients were obtained in
\cite{neuenkirch2009asymptotic,protter2020asymptotic}.
Euler error limits have also been used to establish central limit
theorems for Monte Carlo methods based on Euler approximations,
including statistical Romberg extrapolation
\cite{kebaier2005statistical} and the multilevel Monte Carlo Euler
method \cite{benalaya2015central}.

The appropriate normalization and the limiting process may change with
the driving process, the form of the equation, and the approximation
scheme.  For SDEs driven by pure-jump processes, the asymptotic error
depends on the behavior of the small jumps
\cite{jacod2004euler,wang2015euler}, while for equations driven by
fractional Brownian motion it depends on the Hurst parameter
\cite{hu2016rate,liu2019firstorder,aida2020error}.
For stochastic Volterra equations with fractional kernels, the
normalization depends on the regularity of the kernel
\cite{fukasawa2023limit,nualart2023error}.  Asymptotic error
distributions for Euler approximations of stochastic delay differential
equations were studied in \cite{liu2026asymptotic}.
For stochastic Hamiltonian systems with additive Brownian noise,
Hong, Liang and Sheng \cite{hong2026asymptotic} obtained an asymptotic
error distribution under the normalization \(n\) for the
\(\theta\)-method, including the Euler--Maruyama scheme.  Changes in the
approximation scheme may also remove lower-order error terms, as for
Milstein-type schemes for SDEs with multiplicative Brownian noise
\cite{yan2005asymptotic} and for stochastic Volterra equations
\cite{liu2025limit}.  Fukasawa and Ob{\l}\'oj
\cite{fukasawa2020efficient} characterized the asymptotic distribution
of the Euler--Maruyama error under general random partitions and
constructed hitting-time partitions that reduce the asymptotic error
relative to the equidistant partition.
These results show that the asymptotic behavior of the discretization
error depends on the driving process, the equation, and the
approximation scheme.  We consider below a local linearization
approximation for an SDE with additive Brownian noise.

From now on, \(\sigma\) denotes a fixed matrix in
\(\mathbb R^{d\times q}\), and the exact equation is
\begin{equation}
    dX_t
    =
    f(X_t)\,dt
    +
    \sigma\,dB_t,
    \qquad
    X_0=x_0 .
\end{equation}
Put \(\Sigma=\sigma\sigma^\top\), and denote the \((j,k)\)-component of
\(\Sigma\) by \(\Sigma^{jk}\).  The local linearization approximation
considered in this paper is the process \(\hat X^n\) solving, on each
mesh interval,
\begin{equation}\label{eq:intro_ll_scheme}
\begin{aligned}
    d\hat X_t^{n,i}
    ={}&
    \Bigg\{
        f^i(\hat X^n_{\varphi_n(t)})
        +
        \sum_{j=1}^d
        \partial_j f^i(\hat X^n_{\varphi_n(t)})
        \bigl(
            \hat X_t^{n,j}
            -
            \hat X_{\varphi_n(t)}^{n,j}
        \bigr)
        \\
    &\hspace{28mm}
        +
        \frac12
        \sum_{j=1}^d\sum_{k=1}^d
        \Sigma^{jk}
        \partial_{jk}f^i(\hat X^n_{\varphi_n(t)})
        \bigl(t-\varphi_n(t)\bigr)
    \Bigg\}\,dt
    \\
    &\quad
    +
    \sum_{\alpha=1}^q
    \sigma^i_\alpha\,dB_t^\alpha,
\end{aligned}
\end{equation}
for \(i=1,\ldots,d\), with \(\hat X_0^n=x_0\).  The assumptions and
the precise definition are given in
\cref{sec:setting_main_result}.  On each mesh interval, the drift is
replaced by an affine function of the current state, with the
coefficients evaluated at the left endpoint.  The final term in the
braces is the conditional expectation of the quadratic Brownian term
in the corresponding second-order local expansion.

Local linearization schemes go back to Ozaki's work on stochastic
modelling and nonlinear time series
\cite{ozaki19852,ozaki1992bridgea}, and were later developed for SDEs,
including multidimensional equations with additive noise
\cite{biscay1996local,shoji1998approximation,jimenez1999simulation}.
They have also been used for simulation and likelihood-based estimation
of discretely observed diffusions
\cite{shoji1998estimation,shoji1998statistical}.
In the additive Brownian setting considered here,
\eqref{eq:intro_ll_scheme} is a local linearization scheme of strong
order \(3/2\).  The strong error bound of order \(n^{-3/2}\) was
established in \cite{jimenez2002approximation}.  Numerical evaluation
of the local linear step and convergence of the implemented schemes
were studied in \cite{jimenez2012convergence}.  Higher-order and weak
local linearization schemes were developed in
\cite{delacruzcancino2010high,carbonell2006weak,
jimenez2015convergence}.

The strong error bound of order \(n^{-3/2}\) suggests the normalization
\(n^{3/2}\), but does not identify the limit of the normalized error.
In related settings, the cancellation of lower-order error terms leads
to larger normalizations.  This occurs for the \(\theta\)-method applied
to stochastic Hamiltonian systems with additive Brownian noise
\cite{hong2026asymptotic}, and for Milstein-type schemes for SDEs with
multiplicative Brownian noise
\cite{yan2005asymptotic}, for which \(n\)-scale error limits were obtained.
This motivates studying the \(n^{3/2}\)-normalized error of
\eqref{eq:intro_ll_scheme}.

The main result of this paper identifies the limit under the
normalization \(n^{3/2}\).  More precisely,
\[
    n\sqrt n\,(X-\hat X^n)
\]
converges stably in \(C([0,1],\mathbb R^d)\) to the solution of a
linear stochastic differential equation; the precise statement is
given in \cref{thm:ll_brownian_stable}.  In contrast to the
Euler--Maruyama limit for the multiplicative SDE
\eqref{eq:intro_em_reference_sde}, the limiting equation contains no
stochastic integral with respect to the original Brownian motion.  Its
martingale term is driven by a Brownian motion independent of the
original \(\sigma\)-field; see
\eqref{eq:ll_limit_sde_stable}.

An important example is overdamped Langevin dynamics with a constant
diffusion matrix.  In this case, the drift may have superlinear growth
and need not satisfy a global linear growth condition.  This motivates
the Lyapunov condition imposed in \cref{sec:setting_main_result}.
The main theorem concerns a fixed finite time interval and does not
address the approximation of invariant measures or long-time Monte
Carlo error.  It identifies the finite-time
discretization error of the local linearization approximation to the
Langevin diffusion.

The proof uses two standard stable convergence results.  After a Taylor
expansion on each mesh interval, the scaled error is written as the solution
of a linear integral equation with an inhomogeneous term \(Z^n\).  We first
identify the stable limit of \(Z^n\).  Its principal part is reduced, by
It\^o's formula and stochastic Fubini's theorem, to a continuous local
martingale, and \cite[Theorem IX.7.3]{jacod2003limit} applies after verifying
the convergence of its bracket and the asymptotic orthogonality conditions.
The remaining terms in \(Z^n\) are negligible in ucp.  The convergence of the
scaled error process then follows from \cite[Theorem 2.5(c)]{jacod1998asymptotic},
which yields the stable convergence of the solution of the corresponding linear
equation.

The paper is organized as follows.  In \cref{sec:setting_main_result} we state
the assumptions, define the approximation and formulate the main theorem.  In
\cref{sec:linear_reduction} we derive the linear equation for the scaled
error and reduce the proof to the stable convergence of its inhomogeneous
term.  In \cref{sec:brownian_forcing} we prove this stable convergence by
approximating the principal term by a continuous local martingale and
controlling the remainders.  The
proof of the main theorem is completed in \cref{sec:proof_main_theorem}.
Auxiliary localization and averaging results are collected in
\cref{app:localization_averaging}, and the estimates for the remainder terms
are given in \cref{app:ll_technical_estimates}.

\section{Setting and Main Result}
\label{sec:setting_main_result}

\subsection{Setting}

Let \(d,q\in\mathbb N\).  Throughout the paper, we use Einstein's
summation convention: an index repeated once in an upper position and once in
a lower position is summed over its full range.
% No summation is understood for indices repeated in the same position unless an explicit summation sign is written. 
Latin indices such as \(i,j,k,\ell\) range over \(1,\ldots,d\),
and Greek indices such as \(\alpha,\beta,\gamma,\delta\) range over \(1,\ldots,q\).

Let
\[
    \sigma=(\sigma^i_\alpha)\in\mathbb R^{d\times q},
    \qquad
    \Sigma\coloneqq \sigma\sigma^\top,
    \qquad
    \Sigma^{jk}\coloneqq
    \sum_{\alpha=1}^q
    \sigma^j_\alpha\sigma^k_\alpha .
\]
Let \(B=(B^\alpha)_{1\le \alpha\le q}\) be a \(q\)-dimensional standard
Brownian motion on a filtered probability space
\[
    (\Omega,\mathcal F,(\mathcal F_t)_{0\le t\le1},\mathbb P)
\]
satisfying the usual conditions, and assume that
\(\mathcal F=\mathcal F_1\).
Let \(f=(f^i):\mathbb R^d\to\mathbb R^d\) belong to
\(C^2(\mathbb R^d;\mathbb R^d)\).  We write
\[
    f^i_j\coloneqq \partial_j f^i,
    \qquad
    f^i_{jk}\coloneqq \partial_{jk} f^i.
\]
When \(f\in C^3(\mathbb R^d;\mathbb R^d)\), we also write
\[
    f^i_{jk\ell}\coloneqq \partial_{jk\ell}f^i .
\]
Assume that there exists a constant \(K_f<\infty\) such that
\begin{equation}\label{eq:radial_growth_condition}
    \langle x,f(x)\rangle
    \le
    K_f(1+|x|^2),
    \qquad
    x\in\mathbb R^d .
\end{equation}
This is a one-sided Lyapunov growth condition and does not impose linear
growth on \(f\).

For \(n\in\mathbb N\), define the left-endpoint map by
\[
    \varphi_n(t)
    \coloneqq
    \frac{\lfloor nt\rfloor}{n}.
\]
All convergence statements in spaces \(C([0,1],\mathbb R^m)\) are understood
with the uniform topology.

We consider the SDE with additive Brownian noise 
\begin{equation}\label{eq:brownian_sde}
    dX_t^i
    =
    f^i(X_t)\,dt
    +
    \sigma^i_\alpha\,dB_t^\alpha,
    \qquad
    X_0=x_0\in\mathbb R^d .
\end{equation}
Since \(f\) is locally Lipschitz, local strong existence and pathwise
uniqueness hold.  Moreover, \eqref{eq:radial_growth_condition} is the
Lyapunov condition associated with \(\mathcal V(x)=1+|x|^2\), because the
generator \(\mathcal L\) of \eqref{eq:brownian_sde} satisfies
\[
    \mathcal L\mathcal V(x)
    =
    2\langle x,f(x)\rangle
    +
    \operatorname{Tr}(\Sigma).
\]
Thus \(\mathcal L\mathcal V(x)\le C\mathcal V(x)\) for some \(C<\infty\).
Hence the standard non-explosion criterion yields a unique non-explosive
strong solution \(X\); see, for instance,
\cite{khasminskii2012stochastic}.

The local linearization approximation \(\hat X^n\) is defined by
\begin{equation}\label{eq:ll_brownian_scheme}
\begin{aligned}
    d\hat X_t^{n,i}
    ={}&
    \Bigl\{
        f^i(\hat X^n_{\varphi_n(t)})
        +
        f^i_j(\hat X^n_{\varphi_n(t)})
        \bigl(
            \hat X_t^{n,j}
            -
            \hat X_{\varphi_n(t)}^{n,j}
        \bigr)
        \\
    &\quad
        +
        \frac12
        \Sigma^{jk}
        f^i_{jk}(\hat X^n_{\varphi_n(t)})
        \bigl(t-\varphi_n(t)\bigr)
    \Bigr\}dt
    +
    \sigma^i_\alpha\,dB_t^\alpha,
\end{aligned}
\end{equation}
where \(\hat X_0^n=x_0\).
On each mesh interval, once the left-endpoint value is fixed, \eqref{eq:ll_brownian_scheme} is an
inhomogeneous linear equation with constant coefficients.

We define the scaled error process by
\begin{equation}
    \hat U_t^{n,i}
    \coloneqq
    n\sqrt n
    \bigl(
        X_t^i-\hat X_t^{n,i}
    \bigr),
    \qquad
    0\le t\le1 .
\end{equation}

\subsection{Main result}

Stable convergence is understood with respect to the original
\(\sigma\)-field \(\mathcal F\).  Let \(E\) be a Polish space.  If \(Y^n\)
are \(E\)-valued random variables defined on the original probability space
and \(Y\) is an \(E\)-valued random variable defined on an extension
\[
    (\widetilde\Omega,\widetilde{\mathcal F},
    \widetilde{\mathbb P})
\]
of \((\Omega,\mathcal F,\mathbb P)\), then
\[
    Y^n\Rightarrow^{\mathrm{st}}Y
\]
means that
\[
    \mathbb E\left[
        \xi\,\Phi(Y^n)
    \right]
    \longrightarrow
    \widetilde{\mathbb E}\left[
        \xi\,\Phi(Y)
    \right]
\]
for every bounded \(\mathcal F\)-measurable random variable \(\xi\) and every
bounded continuous function \(\Phi:E\to\mathbb R\),
where \(\xi\) is identified with its copy on the extension.

\begin{theorem}[Stable convergence of the local linearization error]
\label{thm:ll_brownian_stable}
Under the standing assumptions, assume additionally that
\[
    f\in C^3(\mathbb R^d;\mathbb R^d).
\]
Let \(W=(W^{\alpha\beta})_{1\le\alpha,\beta\le q}\) be a
\(q^2\)-dimensional standard Brownian motion independent of \(\mathcal F\),
defined on an extension of the original probability space.  Let \(U\) be the
unique continuous solution of
\begin{equation}\label{eq:ll_limit_sde_stable}
    dU_t^i
    =
    f^i_j(X_t)U_t^j\,dt
    +
    \frac{1}{4\sqrt3}
    f^i_{jk}(X_t)
    \sigma^j_\alpha
    \sigma^k_\beta
    \bigl(
        dW_t^{\alpha\beta}
        +
        dW_t^{\beta\alpha}
    \bigr),
    \qquad
    U_0=0 .
\end{equation}
Then
\begin{equation}\label{eq:ll_brownian_stable_convergence}
    \hat U^n
    \Rightarrow^{\mathrm{st}}
    U
    \qquad
    \text{in } C([0,1],\mathbb R^d).
\end{equation}
\end{theorem}

% \begin{remark}[Joint weak convergence]
% The conclusion of \cref{thm:ll_brownian_stable} implies joint weak convergence
% with any random element defined on the original probability space.  In
% particular,
% \[
%     (B,\hat U^n)
%     \Rightarrow
%     (B,U)
%     \qquad
%     \mbox{in }
%     C([0,1],\mathbb R^q)
%     \times
%     C([0,1],\mathbb R^d).
% \]
% The Brownian motion \(W\) is independent of \(\mathcal F\), but \(U\) is not
% independent of \(B\) in general, because the coefficients in
% \eqref{eq:ll_limit_sde_stable} depend on \(X\).
% \end{remark}

% \begin{remark}[Conditional Gaussianity]
% For each fixed \(t\), the random vector \(U_t\) is Gaussian conditionally on
% \(\mathcal F\).  More generally, conditionally on \(\mathcal F\), the process
% \(U\) is the solution of a linear equation driven by the Brownian motion \(W\)
% with deterministic, path-dependent coefficients.
% \end{remark}

\section{Reduction to a Linear Equation}
\label{sec:linear_reduction}

We reduce the proof of \cref{thm:ll_brownian_stable} to the stable convergence
of the inhomogeneous term \(Z^n\).  Since the Brownian noise is additive and
is the same in the equation and in the scheme, the error equation has finite
variation.

\subsection{Error equation and decomposition of the Inhomogeneous Term}

Put
\[
    \eta_s^{n,i}
    \coloneqq
    \sigma^i_\alpha
    \bigl(
        B_s^\alpha-B_{\varphi_n(s)}^\alpha
    \bigr),
    \qquad
    \theta_s^{n,i}
    \coloneqq
    \hat X_s^{n,i}
    -
    \hat X_{\varphi_n(s)}^{n,i}
    -
    \eta_s^{n,i}.
\]
For \(1\le i\le d\), define
\begin{align}
    \widetilde R_s^{n,i}
    &\coloneqq
    \int_0^1(1-\lambda)
    \Bigl[
        f^i_{jk}
        \bigl(
            \hat X_{\varphi_n(s)}^n
            +
            \lambda
            (X_s-\hat X_{\varphi_n(s)}^n)
        \bigr)
        -
        f^i_{jk}
        \bigl(
            \hat X_{\varphi_n(s)}^n
        \bigr)
    \Bigr] d\lambda
    \notag\\
    &\hspace{2.5cm}\times
    \bigl(
        X_s^j-\hat X_{\varphi_n(s)}^{n,j}
    \bigr)
    \bigl(
        X_s^k-\hat X_{\varphi_n(s)}^{n,k}
    \bigr).
    \label{eq:ll_taylor_remainder_def}
\end{align}

\begin{proposition}[Linear error equation]
\label[proposition]{prop:ll_exact_linear_reduction}
The scaled error satisfies
\begin{equation}\label{eq:ll_exact_linear_equation}
    \hat U_t^{n,i}
    =
    Z_t^{n,i}
    +
    \int_0^t
    f^i_j(\hat X_{\varphi_n(s)}^n)
    \hat U_s^{n,j}\,ds,
    \qquad
    0\le t\le1,
\end{equation}
where
\begin{equation}\label{eq:ll_Zn_full_decomposition_reduction}
    Z^n
    \coloneqq
    \bar M^n
    +
    \sum_{\ell=1}^5 R^{(\ell),n}.
\end{equation}
Here
\begin{align}
    \bar M_t^{n,i}
    &\coloneqq
    \frac{n\sqrt n}{2}
    \int_0^t
    f^i_{jk}(\hat X_{\varphi_n(s)}^n)
    \sigma^j_\alpha\sigma^k_\beta
    \\
    &\quad\times
    \biggl[
        \bigl(
            B_s^\alpha-B_{\varphi_n(s)}^\alpha
        \bigr)
        \bigl(
            B_s^\beta-B_{\varphi_n(s)}^\beta
        \bigr)
        -
        \delta^{\alpha\beta}
        \bigl(s-\varphi_n(s)\bigr)
    \biggr] ds,
    \\
    R_t^{(1),n,i}
    &\coloneqq
    n\sqrt n
    \int_0^t
    \widetilde R_s^{n,i}\,ds,
    \\
    R_t^{(2),n,i}
    &\coloneqq
    n\sqrt n
    \int_0^t
    f^i_{jk}(\hat X_{\varphi_n(s)}^n)
    \bigl(
        \hat X_s^{n,j}
        -
        \hat X_{\varphi_n(s)}^{n,j}
    \bigr)
    \bigl(
        X_s^k-\hat X_s^{n,k}
    \bigr) ds,
    \\
    R_t^{(3),n,i}
    &\coloneqq
    \frac{n\sqrt n}{2}
    \int_0^t
    f^i_{jk}(\hat X_{\varphi_n(s)}^n)
    \bigl(
        X_s^j-\hat X_s^{n,j}
    \bigr)
    \bigl(
        X_s^k-\hat X_s^{n,k}
    \bigr) ds,
    \\
    R_t^{(4),n,i}
    &\coloneqq
    \frac{n\sqrt n}{2}
    \int_0^t
    f^i_{jk}(\hat X_{\varphi_n(s)}^n)
    \theta_s^{n,j}\theta_s^{n,k}\,ds,
    \\
    R_t^{(5),n,i}
    &\coloneqq
    n\sqrt n
    \int_0^t
    f^i_{jk}(\hat X_{\varphi_n(s)}^n)
    \theta_s^{n,j}\eta_s^{n,k}\,ds.
\end{align}
\end{proposition}

\begin{proof}
Subtracting \eqref{eq:ll_brownian_scheme} from
\eqref{eq:brownian_sde} gives
\[
\begin{aligned}
    X_t^i-\hat X_t^{n,i}
    ={}&
    \int_0^t
    \Bigl[
        f^i(X_s)
        -
        f^i(\hat X_{\varphi_n(s)}^n)
        \\
    &
        -
        f^i_j(\hat X_{\varphi_n(s)}^n)
        \bigl(
            \hat X_s^{n,j}
            -
            \hat X_{\varphi_n(s)}^{n,j}
        \bigr)
        -
        \frac12
        \Sigma^{jk}
        f^i_{jk}(\hat X_{\varphi_n(s)}^n)
        \bigl(s-\varphi_n(s)\bigr)
    \Bigr] ds .
\end{aligned}
\]
Taylor's formula at \(\hat X_{\varphi_n(s)}^n\), with the remainder
\eqref{eq:ll_taylor_remainder_def}, yields
\[
\begin{aligned}
    f^i(X_s)
    &=
    f^i(\hat X_{\varphi_n(s)}^n)
    +
    f^i_j(\hat X_{\varphi_n(s)}^n)
    \bigl(
        \hat X_s^{n,j}
        -
        \hat X_{\varphi_n(s)}^{n,j}
        +
        X_s^j-\hat X_s^{n,j}
    \bigr)
    \\
    &\quad
    +
    \frac12
    f^i_{jk}(\hat X_{\varphi_n(s)}^n)
    \bigl(
        \hat X_s^{n,j}
        -
        \hat X_{\varphi_n(s)}^{n,j}
        +
        X_s^j-\hat X_s^{n,j}
    \bigr)
    \\
    &\qquad\times
    \bigl(
        \hat X_s^{n,k}
        -
        \hat X_{\varphi_n(s)}^{n,k}
        +
        X_s^k-\hat X_s^{n,k}
    \bigr)
    +
    \widetilde R_s^{n,i}.
\end{aligned}
\]
After substitution, the first-order error term gives the linear part in
\eqref{eq:ll_exact_linear_equation}.  Multiplying by \(n\sqrt n\), the
remaining terms are
\[
\begin{aligned}
    &\frac{n\sqrt n}{2}
    \int_0^t
    f^i_{jk}(\hat X_{\varphi_n(s)}^n)
    \left[
        \bigl(
            \hat X_s^{n,j}
            -
            \hat X_{\varphi_n(s)}^{n,j}
        \bigr)
        \bigl(
            \hat X_s^{n,k}
            -
            \hat X_{\varphi_n(s)}^{n,k}
        \bigr)
        -
        \Sigma^{jk}
        \bigl(s-\varphi_n(s)\bigr)
    \right] ds
    \\
    &\quad
    +
    R_t^{(1),n,i}
    +
    R_t^{(2),n,i}
    +
    R_t^{(3),n,i}.
\end{aligned}
\]
Using
\[
    \hat X_s^{n,i}
    -
    \hat X_{\varphi_n(s)}^{n,i}
    =
    \eta_s^{n,i}+\theta_s^{n,i},
\]
and the symmetry of \(f^i_{jk}\) in \(j,k\), the centered quadratic term
decomposes into \(\bar M^{n,i}+R^{(4),n,i}+R^{(5),n,i}\).  This proves
\eqref{eq:ll_exact_linear_equation} and
\eqref{eq:ll_Zn_full_decomposition_reduction}.
\end{proof}

\subsection{Coefficient convergence and stability of the linear equation}

The coefficient in the linear equation is asymptotically obtained by replacing
\(\hat X_{\varphi_n(\cdot)}^n\) with \(X\).

\begin{proposition}[Coefficient convergence]
\label[proposition]{prop:ll_coefficients}
For every \(1\le i,j\le d\), the sequence
\[
    \sup_{t\le1}
    |f^i_j(\hat X_{\varphi_n(t)}^n)|
\]
is tight.  Moreover,
\begin{equation}\label{eq:ll_Vn_ucp_conv}
    \int_0^\cdot
    f^i_j(\hat X_{\varphi_n(s)}^n)\,ds
    \xrightarrow{\mathrm{ucp}}
    \int_0^\cdot
    f^i_j(X_s)\,ds .
\end{equation}
\end{proposition}

\begin{proof}
The tightness follows from \cref{le:ll_consistency} and the local boundedness of
\(f^i_j\).  Indeed, for every \(\varepsilon>0\), there exists \(\rho>0\)
such that, for all sufficiently large \(n\),
\[
    \mathbb P
    \left(
        \sup_{t\le1}
        \bigl(
            |X_t|\vee|\hat X_t^n|
        \bigr)>\rho
    \right)
    <\varepsilon .
\]
On the complement of this event,
\[
    \sup_{t\le1}
    |f^i_j(\hat X_{\varphi_n(t)}^n)|
    \le
    \sup_{|x|\le\rho}|f^i_j(x)|<\infty .
\]
Thus the sequence is tight.

Let \(\omega_\rho\) be a modulus of continuity of \(f^i_j\) on
\(\{|x|\le \rho\}\).  On the event
\[
    \sup_{s\le1}
    \bigl(
        |X_s|\vee|\hat X_s^n|
    \bigr)
    \le \rho ,
\]
we have
\[
\begin{aligned}
    \sup_{t\le1}
    \left|
        \int_0^t
        \bigl(
            f^i_j(\hat X_{\varphi_n(s)}^n)
            -
            f^i_j(X_s)
        \bigr)ds
    \right|
    &\le
    \int_0^1
    \left|
        f^i_j(\hat X_{\varphi_n(s)}^n)
        -
        f^i_j(X_s)
    \right|ds
    \\
    &\le
    \omega_\rho
    \left(
        \sup_{s\le1}
        |\hat X_{\varphi_n(s)}^n-X_s|
    \right)
    \\
    &\le
    \omega_\rho
    \left(
        \sup_{s\le1}
        |\hat X_s^n-X_s|
        +
        \sup_{s\le1}
        |\hat X_s^n-\hat X_{\varphi_n(s)}^n|
    \right).
\end{aligned}
\]
The quantity inside \(\omega_\rho\) tends to zero in probability by
\cref{le:ll_consistency}.  Letting \(\rho\uparrow\infty\) and using
\cref{le:ll_outer_localization} proves \eqref{eq:ll_Vn_ucp_conv}.
\end{proof}

Let \(W\) be the Brownian motion introduced in
\cref{thm:ll_brownian_stable}.  Define the continuous process \(Z\) on the same
extension by
\begin{equation}\label{eq:ll_Z_limit_reduction_section}
    Z_t^i
    \coloneqq
    \frac{1}{4\sqrt3}
    \int_0^t
    f^i_{jk}(X_s)
    \sigma^j_\alpha\sigma^k_\beta
    \bigl(
        dW_s^{\alpha\beta}
        +
        dW_s^{\beta\alpha}
    \bigr),
    \qquad
    Z_0=0 .
\end{equation}

\begin{proposition}[Reduction to the convergence of \(Z^n\)]
\label[proposition]{prop:ll_reduction_to_forcing}
Assume that
\begin{equation}\label{eq:ll_forcing_stable_input_reduction}
    Z^n
    \Rightarrow^{\mathrm{st}}
    Z
    \qquad
    \text{in }C([0,1],\mathbb R^d) .
\end{equation}
Then
\[
    \hat U^n
    \Rightarrow^{\mathrm{st}}
    U
    \qquad
    \text{in }C([0,1],\mathbb R^d),
\]
where \(U\) is the unique continuous solution of
\begin{equation}\label{eq:ll_limit_sde_reduction_section}
    U_t^i
    =
    Z_t^i
    +
    \int_0^t
    f^i_j(X_s)U_s^j\,ds,
    \qquad
    U_0=0.
\end{equation}
Equivalently,
\begin{equation}\label{eq:ll_limit_sde_reduction_section_differential}
    dU_t^i
    =
    f^i_j(X_t)U_t^j\,dt
    +
    \frac{1}{4\sqrt3}
    f^i_{jk}(X_t)
    \sigma^j_\alpha\sigma^k_\beta
    \bigl(
        dW_t^{\alpha\beta}
        +
        dW_t^{\beta\alpha}
    \bigr),
    \qquad
    U_0=0.
\end{equation}
\end{proposition}

\begin{proof}
For \(1\le i,j\le d\), set
\[
    A_{t,j}^{n,i}
    \coloneqq
    \int_0^t
    f^i_j(\hat X_{\varphi_n(s)}^n)\,ds,
    \qquad
    A_{t,j}^i
    \coloneqq
    \int_0^t
    f^i_j(X_s)\,ds .
\]
By \cref{prop:ll_coefficients},
\[
    A^n
    \xrightarrow{\mathrm{ucp}}
    A,
    \qquad
    \sup_{t\le1}
    \left\|
        \bigl(
            f^i_j(\hat X_{\varphi_n(t)}^n)
        \bigr)_{i,j}
    \right\|
    \quad
    \text{is tight}.
\]
For \(m=1,\ldots,n\), define
\[
    H_s^n
    \coloneqq
    \bigl(
        f^i_j(\hat X_{(m-1)/n}^n)
    \bigr)_{i,j},
    \qquad
    s\in((m-1)/n,m/n],
\]
and put
\[
    H_0^n
    \coloneqq
    \bigl(
        f^i_j(x_0)
    \bigr)_{i,j}.
\]
Then \(H^n\) is predictable.  Changing the coefficient on the grid points
affects neither \(A^n\) nor \eqref{eq:ll_exact_linear_equation}; hence
\[
    A_t^n
    =
    \int_0^t H_s^n\,ds,
\]
and \(\sup_{s\le1}\|H_s^n\|\) is tight.

By \eqref{eq:ll_forcing_stable_input_reduction} and
\cite[Theorem 3.18(b)]{hausler2015stable},
\[
    (Z^n,A^n)
    \Rightarrow^{\mathrm{st}}
    (Z,A)
    \qquad
    \text{in }
    C([0,1],\mathbb R^d)\times C([0,1],\mathbb R^{d\times d}).
\]
Moreover, \eqref{eq:ll_exact_linear_equation} can be written as
\[
    \hat U_t^{n,i}
    =
    Z_t^{n,i}
    +
    \int_0^t
    \hat U_s^{n,j}\,dA_{s,j}^{n,i}.
\]
Since \(\sup_{s\le1}\|H_s^n\|\) is tight, we may apply
\cite[Theorem 2.5(c)]{jacod1998asymptotic} with \(Y_t=t\),
\(J^n=Z^n\), and \(V^n=A^n\).  It follows that
\[
    \hat U^n
    \Rightarrow^{\mathrm{st}}
    U,
\]
where
\[
    U_t^i
    =
    Z_t^i
    +
    \int_0^t
    U_s^j\,dA_{s,j}^i
    =
    Z_t^i
    +
    \int_0^t
    f^i_j(X_s)U_s^j\,ds .
\]
\end{proof}

It remains to prove \eqref{eq:ll_forcing_stable_input_reduction}.  In view of
\eqref{eq:ll_Zn_full_decomposition_reduction}, this amounts to proving the
stable convergence of \(\bar M^n\) and the ucp convergence of
\(R^{(\ell),n}\) to zero for \(\ell =1,2,3,4,5\).

\section{Stable Limit of the Inhomogeneous Term}
\label{sec:brownian_forcing}

We prove the convergence required in
\cref{prop:ll_reduction_to_forcing}.  By
\eqref{eq:ll_Zn_full_decomposition_reduction},
\begin{equation}\label{eq:forcing_decomp_brownian_section}
    Z^n
    =
    \bar M^n
    +
    \sum_{\ell=1}^5 R^{(\ell),n}.
\end{equation}
The principal term is first approximated by a continuous local martingale,
to which a martingale convergence theorem is applied.  The remaining terms
are estimated in \cref{app:ll_technical_estimates}.
The tensors
\[
    \bigl(
        f^i_{jk}(\hat X_{\varphi_n(s)}^n)
        \sigma^j_\alpha\sigma^k_\beta
    \bigr)_{\alpha,\beta}
    \quad\mbox{and}\quad
    \bigl(
        f^i_{jk}(X_s)
        \sigma^j_\alpha\sigma^k_\beta
    \bigr)_{\alpha,\beta}
\]
are symmetric in \(\alpha,\beta\).

\subsection{Martingale approximation of the principal term}

Define the continuous local martingale
\begin{equation}
\begin{aligned}
    \widetilde M_t^{n,i}
    \coloneqq{}&
    \frac{\sqrt n}{2}
    \int_0^t
    f^i_{jk}(\hat X_{\varphi_n(s)}^n)
    \sigma^j_\alpha\sigma^k_\beta
    \bigl(1+\lfloor ns\rfloor-ns\bigr)
    \\
    &\quad\times
    \left(
        \bigl(B_s^\alpha-B_{\varphi_n(s)}^\alpha\bigr)dB_s^\beta
        +
        \bigl(B_s^\beta-B_{\varphi_n(s)}^\beta\bigr)dB_s^\alpha
    \right).
\end{aligned}
\end{equation}

\begin{lemma}[Negligibility of the last mesh interval]
\label[lemma]{le:ll_principal_boundary_new}
We have
\begin{equation}
    \bar M^n-\widetilde M^n
    \xrightarrow{\mathrm{ucp}}0 .
\end{equation}
\end{lemma}

\begin{proof}
It suffices to prove the assertion after localization, under the additional
assumption that
\((f^i_{jk}(\hat X_{\varphi_n(\cdot)}^n)
\sigma^j_\alpha\sigma^k_\beta)_{i,\alpha,\beta}\) is uniformly bounded.  The
localization is removed at the end by \cref{le:ll_outer_localization}.

By It\^o's formula,
\[
\begin{aligned}
    &\bigl(B_s^\alpha-B_{\varphi_n(s)}^\alpha\bigr)
    \bigl(B_s^\beta-B_{\varphi_n(s)}^\beta\bigr)
    -
    \delta^{\alpha\beta}(s-\varphi_n(s))
    \\
    &\quad=
    \int_{\varphi_n(s)}^s
    \bigl(B_u^\alpha-B_{\varphi_n(s)}^\alpha\bigr)dB_u^\beta
    +
    \int_{\varphi_n(s)}^s
    \bigl(B_u^\beta-B_{\varphi_n(s)}^\beta\bigr)dB_u^\alpha .
\end{aligned}
\]
Substitution into \(\bar M^n\), followed by stochastic Fubini, gives
\[
\begin{aligned}
    \bar M_t^{n,i}
    ={}&
    \frac{n\sqrt n}{2}
    \int_0^t
    f^i_{jk}(\hat X_{\varphi_n(u)}^n)
    \sigma^j_\alpha\sigma^k_\beta
    \left(
        ((\varphi_n(u)+n^{-1})\wedge t)-u
    \right)
    \\
    &\quad\times
    \left(
        \bigl(B_u^\alpha-B_{\varphi_n(u)}^\alpha\bigr)dB_u^\beta
        +
        \bigl(B_u^\beta-B_{\varphi_n(u)}^\beta\bigr)dB_u^\alpha
    \right).
\end{aligned}
\]
If \(u\le t\) and \(u\) does not belong to the last mesh interval intersecting
\(t\), the factor in the preceding display equals
\((1+\lfloor nu\rfloor-nu)/n\).  Hence
\(\bar M^n-\widetilde M^n\) is given only by the contribution from that last
interval.

Let \(t_m\coloneqq m/n\) and \(I_m\coloneqq((m-1)/n,m/n]\).  For
\(t\in I_m\),
\[
\begin{aligned}
    \bar M_t^{n,i}-\widetilde M_t^{n,i}
    ={}&
    -
    \frac{n\sqrt n}{2}
    (t_m-t)
    \int_{(m-1)/n}^t
    f^i_{jk}(\hat X_{\varphi_n(u)}^n)
    \sigma^j_\alpha\sigma^k_\beta
    \\
    &\quad\times
    \left(
        \bigl(B_u^\alpha-B_{\varphi_n(u)}^\alpha\bigr)dB_u^\beta
        +
        \bigl(B_u^\beta-B_{\varphi_n(u)}^\beta\bigr)dB_u^\alpha
    \right).
\end{aligned}
\]
Thus, for every \(p>2\), the Burkholder--Davis--Gundy inequality and Brownian
scaling yield
\[
    \mathbb E
    \left[
        \sup_{t\in I_m}
        |\bar M_t^{n,i}-\widetilde M_t^{n,i}|^p
    \right]
    \le
    K_p n^{-p/2}.
\]
Summing over \(m\) gives
\[
    \mathbb P
    \left(
        \sup_{t\le1}
        |\bar M_t^{n,i}-\widetilde M_t^{n,i}|>\varepsilon
    \right)
    \le
    K_{p,\varepsilon} n^{1-p/2}
    \longrightarrow0 .
\]
The localized convergence is therefore ucp.  Let
\(\chi_\rho\in C_c^\infty(\mathbb R^d)\) be equal to \(1\) on
\(\{|x|\le\rho\}\), and repeat the preceding argument with
\(f^i_{jk}(x)\) replaced by \(\chi_\rho(x)f^i_{jk}(x)\).  On the event
\[
    \sup_{t\le1}|\hat X_t^n|\le\rho,
\]
the cutoff and original processes coincide.  Letting \(\rho\uparrow\infty\)
and using \cref{le:ll_outer_localization} proves the assertion.
\end{proof}

The following averaging estimate is used in the computation of the bracket of
\(\widetilde M^n\).

\begin{lemma}[Brownian weighted averaging]
\label[lemma]{le:brownian_weighted_averaging}
Let \(G:\mathbb R^d\to\mathbb R\) be continuous.  Then, for
\(1\le\alpha,\gamma\le q\),
\begin{equation}
\begin{aligned}
    &n
    \int_0^\cdot
    G(\hat X_{\varphi_n(s)}^n)
    \bigl(1+\lfloor ns\rfloor-ns\bigr)^2
    \bigl(B_s^\alpha-B_{\varphi_n(s)}^\alpha\bigr)
    \bigl(B_s^\gamma-B_{\varphi_n(s)}^\gamma\bigr)ds
    \\
    &\qquad\qquad\qquad
    \xrightarrow{\mathrm{ucp}}
    \frac{\delta^{\alpha\gamma}}{12}
    \int_0^\cdot G(X_s)\,ds .
\end{aligned}
\end{equation}
\end{lemma}

\begin{proof}
By localization and \cref{le:ll_outer_localization}, it suffices to consider
bounded uniformly continuous \(G\).  Then
\[
    \int_0^1
    |G(\hat X_{\varphi_n(s)}^n)-G(X_s)|\,ds
    \xrightarrow{P}0
\]
by \cref{le:ll_consistency}.

Put
\[
\begin{aligned}
    Y_t^n
    \coloneqq{}&
    n
    \int_0^t
    G(\hat X_{\varphi_n(s)}^n)
    \bigl(1+\lfloor ns\rfloor-ns\bigr)^2
    \\
    &\quad\times
    \left(
        \bigl(B_s^\alpha-B_{\varphi_n(s)}^\alpha\bigr)
        \bigl(B_s^\gamma-B_{\varphi_n(s)}^\gamma\bigr)
        -
        \delta^{\alpha\gamma}(s-\varphi_n(s))
    \right)ds .
\end{aligned}
\]
The sequence \((Y^n_{m/n})_{m=0}^n\) is a discrete-time martingale.
Brownian scaling gives, uniformly in \(m\),
\[
    \mathbb E
    \left[
        |Y_{m/n}^n-Y_{(m-1)/n}^n|^2
        \,\middle|\,
        \mathcal F_{(m-1)/n}
    \right]
    \le
    K n^{-2},
\]
and
\[
    \mathbb E
    \left[
        \sup_{(m-1)/n\le t\le m/n}
        |Y_t^n-Y_{(m-1)/n}^n|^2
        \,\middle|\,
        \mathcal F_{(m-1)/n}
    \right]
    \le
    K n^{-2}.
\]
Therefore, the discrete-time Doob inequality and summation over the mesh
intervals yield
\[
    \mathbb E
    \left[
        \sup_{t\le1}|Y_t^n|^2
    \right]
    \le
    K n^{-1}.
\]
Hence \(Y^n\to0\) ucp.

It remains to identify the deterministic part.  Since
\[
    n
    \bigl(1+\lfloor ns\rfloor-ns\bigr)^2
    (s-\varphi_n(s))
    =
    (1-u)^2u,
    \qquad
    u\coloneqq ns-\lfloor ns\rfloor,
\]
\cref{le:contitoucp} gives
\[
\begin{aligned}
    &n
    \int_0^\cdot
    G(\hat X_{\varphi_n(s)}^n)
    \bigl(1+\lfloor ns\rfloor-ns\bigr)^2
    (s-\varphi_n(s))\,ds
    \\
    &\qquad\xrightarrow{\mathrm{ucp}}
    \left(
        \int_0^1(1-u)^2u\,du
    \right)
    \int_0^\cdot G(X_s)\,ds .
\end{aligned}
\]
Since
\[
    \int_0^1(1-u)^2u\,du=\frac1{12},
\]
the result follows.
\end{proof}

\begin{proposition}[Stable limit of the principal term]
\label[proposition]{prop:ll_principal_stable_limit_new}
Let \(Z\) be the process defined in
\eqref{eq:ll_Z_limit_reduction_section}.  Then
\begin{equation}\label{eq:ll_principal_stable_convergence_new}
    \bar M^n
    \Rightarrow^{\mathrm{st}}
    Z
\end{equation}
in \(C([0,1],\mathbb R^d)\).
\end{proposition}

\begin{proof}
By \cref{le:ll_principal_boundary_new}, it suffices to prove the assertion with
\(\widetilde M^n\) in place of \(\bar M^n\).  We verify the hypotheses of
\cite[Theorem IX.7.3]{jacod2003limit}.

First, for \(1\le\gamma\le q\),
\[
\begin{aligned}
    \langle \widetilde M^{n,i},B^\gamma\rangle_t
    ={}&
    \frac{\sqrt n}{2}
    \int_0^t
    f^i_{jk}(\hat X_{\varphi_n(s)}^n)
    \sigma^j_\alpha\sigma^k_\beta
    \bigl(1+\lfloor ns\rfloor-ns\bigr)
    \\
    &\quad\times
    \Bigl(
        \bigl(B_s^\alpha-B_{\varphi_n(s)}^\alpha\bigr)\delta^{\beta\gamma}
        +
        \bigl(B_s^\beta-B_{\varphi_n(s)}^\beta\bigr)\delta^{\alpha\gamma}
    \Bigr)ds .
\end{aligned}
\]
Fix \(\rho>0\), let \(\chi_\rho\in C_c^\infty(\mathbb R^d)\) be equal to
\(1\) on \(\{|x|\le\rho\}\), and replace \(f^i_{jk}\) in the preceding
bracket by \(\chi_\rho f^i_{jk}\).  Denote the resulting cutoff bracket by
\(C^{n,i,\gamma,\rho}\).  Its values at the grid points form a discrete-time
martingale.  As in the proof of
\cref{le:brownian_weighted_averaging}, Brownian scaling gives, uniformly in
\(m\),
\[
    \mathbb E
    \left[
        |C_{m/n}^{n,i,\gamma,\rho}
        -C_{(m-1)/n}^{n,i,\gamma,\rho}|^2
        \,\middle|\,
        \mathcal F_{(m-1)/n}
    \right]
    \le
    K_\rho n^{-2},
\]
and
\[
    \mathbb E
    \left[
        \sup_{(m-1)/n\le t\le m/n}
        |C_t^{n,i,\gamma,\rho}
        -C_{(m-1)/n}^{n,i,\gamma,\rho}|^2
        \,\middle|\,
        \mathcal F_{(m-1)/n}
    \right]
    \le
    K_\rho n^{-2}.
\]
The discrete-time Doob inequality and summation over the mesh intervals
therefore give
\[
    \mathbb E
    \left[
        \sup_{t\le1}|C_t^{n,i,\gamma,\rho}|^2
    \right]
    \le
    K_\rho n^{-1}.
\]
Hence the cutoff bracket converges to zero ucp.  On the event
\[
    \sup_{t\le1}|\hat X_t^n|\le\rho,
\]
it coincides with \(\langle \widetilde M^{n,i},B^\gamma\rangle\).  Letting
\(\rho\uparrow\infty\) and using \cref{le:ll_outer_localization}, we obtain
\begin{equation}
    \langle \widetilde M^{n,i},B^\gamma\rangle
    \xrightarrow{\mathrm{ucp}}0 .
\end{equation}

Next,
\begin{align}
    \langle \widetilde M^{n,i},\widetilde M^{n,i'}\rangle_t
    &=
    \frac n4
    \int_0^t
    f^i_{jk}(\hat X_{\varphi_n(s)}^n)
    \sigma^j_\alpha\sigma^k_\beta
    f^{i'}_{\ell m}(\hat X_{\varphi_n(s)}^n)
    \sigma^\ell_\gamma\sigma^m_\delta
    \notag\\
    &\quad\times
    \bigl(1+\lfloor ns\rfloor-ns\bigr)^2
    \Bigl(
        \bigl(B_s^\alpha-B_{\varphi_n(s)}^\alpha\bigr)
        \bigl(B_s^\gamma-B_{\varphi_n(s)}^\gamma\bigr)
        \delta^{\beta\delta}
        \notag\\
    &\qquad
        +
        \bigl(B_s^\alpha-B_{\varphi_n(s)}^\alpha\bigr)
        \bigl(B_s^\delta-B_{\varphi_n(s)}^\delta\bigr)
        \delta^{\beta\gamma}
        \notag\\
    &\qquad
        +
        \bigl(B_s^\beta-B_{\varphi_n(s)}^\beta\bigr)
        \bigl(B_s^\gamma-B_{\varphi_n(s)}^\gamma\bigr)
        \delta^{\alpha\delta}
        \notag\\
    &\qquad
        +
        \bigl(B_s^\beta-B_{\varphi_n(s)}^\beta\bigr)
        \bigl(B_s^\delta-B_{\varphi_n(s)}^\delta\bigr)
        \delta^{\alpha\gamma}
    \Bigr)ds .
    \label{eq:bracket_Mn_Mn_expansion}
\end{align}
Applying \cref{le:brownian_weighted_averaging} to the continuous coefficient
functions appearing in \eqref{eq:bracket_Mn_Mn_expansion}, and using the
symmetry in \(\alpha,\beta\), yields
\begin{equation}
\begin{aligned}
    \langle \widetilde M^{n,i},\widetilde M^{n,i'}\rangle_t
    \xrightarrow{\mathrm{ucp}}
    \frac1{12}
    \int_0^t
    f^i_{jk}(X_s)
    f^{i'}_{\ell m}(X_s)
    \Sigma^{j\ell}\Sigma^{km}\,ds .
\end{aligned}
\end{equation}

Let \(N\) be a bounded martingale orthogonal to all components of \(B\).
Since \(\widetilde M^n\) is a stochastic integral with respect to \(B\),
\[
    \langle \widetilde M^{n,i},N\rangle=0.
\]
Together with
\[
    \langle \widetilde M^{n,i},B^\gamma\rangle
    \xrightarrow{\mathrm{ucp}}0,
    \qquad
    \gamma=1,\ldots,q,
\]
this verifies the bracket conditions in
\cite[Theorem IX.7.3]{jacod2003limit} with \(Z=B\) and \(G=0\).
The processes \(\widetilde M^n\) are continuous, so the jump condition in
that theorem is automatically satisfied.  Since
\[
    \langle B^\alpha,B^\gamma\rangle_t
    =
    \delta^{\alpha\gamma}t
\]
and the limiting bracket is absolutely continuous,
\cite[Theorem IX.7.3(b)]{jacod2003limit} shows that the stable limit can be
realized on a Wiener extension as a continuous conditionally centered
Gaussian martingale with bracket
\[
\begin{aligned}
    \frac1{12}
    \int_0^t
    f^i_{jk}(X_s)
    f^{i'}_{\ell m}(X_s)
    \Sigma^{j\ell}\Sigma^{km}\,ds,
    \qquad
    1\le i,i'\le d,
\end{aligned}
\]
driven by a Brownian motion independent of \(\mathcal F\).

The process \(Z\) defined in
\eqref{eq:ll_Z_limit_reduction_section} is such a realization.  Indeed, the
independence of the components of \(W\), together with the symmetry of the
coefficient in \(\alpha,\beta\), gives
\[
\begin{aligned}
    \langle Z^i,Z^{i'}\rangle_t
    &=
    4\left(\frac1{4\sqrt3}\right)^2
    \int_0^t
    f^i_{jk}(X_s)
    f^{i'}_{\ell m}(X_s)
    \Sigma^{j\ell}\Sigma^{km}\,ds
    \\
    &=
    \frac1{12}
    \int_0^t
    f^i_{jk}(X_s)
    f^{i'}_{\ell m}(X_s)
    \Sigma^{j\ell}\Sigma^{km}\,ds .
\end{aligned}
\]
Therefore the stable limit may be realized as \(Z\), and hence
\[
    \widetilde M^n
    \Rightarrow^{\mathrm{st}}
    Z .
\]
Combining this with \cref{le:ll_principal_boundary_new} and
\cite[Theorem 3.18(a)]{hausler2015stable} proves
\eqref{eq:ll_principal_stable_convergence_new}.
\end{proof}

\subsection{Remainders and the Stable Limit}

\begin{proposition}[Negligibility of the remainders]
\label[proposition]{prop:ll_remainders_ucp_zero_new}
Assume that \(f\in C^3(\mathbb R^d;\mathbb R^d)\).  Then
\begin{equation}\label{eq:ll_remainders_ucp_zero_new}
    R^{(\ell),n}
    \xrightarrow{\mathrm{ucp}}0,
    \qquad
    \ell=1,\dots,5 .
\end{equation}
\end{proposition}

\begin{proof}
The assertion follows from
\cref{prop:ll_R123_ucp_zero,prop:ll_R45_ucp_zero}.
\end{proof}

\begin{proposition}[Stable convergence of \(Z^n\)]
\label[proposition]{prop:ll_forcing_stable_limit_new}
Assume that \(f\in C^3(\mathbb R^d;\mathbb R^d)\), and let \(Z\) be the
process defined in \eqref{eq:ll_Z_limit_reduction_section}.  Then
\begin{equation}
    Z^n
    \Rightarrow^{\mathrm{st}}
    Z
\end{equation}
in \(C([0,1],\mathbb R^d)\).
\end{proposition}

\begin{proof}
By \cref{prop:ll_principal_stable_limit_new},
\[
    \bar M^n
    \Rightarrow^{\mathrm{st}}
    Z
\]
in \(C([0,1],\mathbb R^d)\).  On the other hand,
\cref{prop:ll_remainders_ucp_zero_new} gives
\[
    \sum_{\ell=1}^5 R^{(\ell),n}
    \xrightarrow{\mathrm{ucp}}0 .
\]
Combining the decomposition \eqref{eq:forcing_decomp_brownian_section} with
\cite[Theorem 3.18(a)]{hausler2015stable} proves
\[
    Z^n
    \Rightarrow^{\mathrm{st}}
    Z
\]
in \(C([0,1],\mathbb R^d)\).
\end{proof}

\section{Proof of the Main Theorem}
\label{sec:proof_main_theorem}

\begin{proof}[Proof of \cref{thm:ll_brownian_stable}]
By \cref{prop:ll_forcing_stable_limit_new},
\[
    Z^n
    \Rightarrow^{\mathrm{st}}
    Z
    \qquad
    \text{in } C([0,1],\mathbb R^d) ,
\]
where \(Z\) is the process defined in
\eqref{eq:ll_Z_limit_reduction_section}.  Hence the hypothesis of
\cref{prop:ll_reduction_to_forcing} is satisfied.  Applying that proposition to
the exact linear equation in \cref{prop:ll_exact_linear_reduction}, we obtain
\[
    \hat U^n
    \Rightarrow^{\mathrm{st}}
    U
    \qquad
    \text{in } C([0,1],\mathbb R^d) ,
\]
where \(U\) is the unique continuous solution of
\[
    U_t^i
    =
    Z_t^i
    +
    \int_0^t f^i_j(X_s)U_s^j\,ds,
    \qquad
    U_0=0.
\]
Substituting the expression for \(Z\) from
\eqref{eq:ll_Z_limit_reduction_section} gives
\[
    dU_t^i
    =
    f^i_j(X_t)U_t^j\,dt
    +
    \frac{1}{4\sqrt3}
    f^i_{jk}(X_t)\sigma^j_\alpha\sigma^k_\beta
    \bigl(
        dW_t^{\alpha\beta}
        +
        dW_t^{\beta\alpha}
    \bigr),
    \qquad
    U_0=0.
\]
Therefore \eqref{eq:ll_brownian_stable_convergence} follows.
\end{proof}

% The following sections are appendices, such as Appendix A, B, C, D, ...

\appendix

\section{Auxiliary Localization and Averaging Lemmas}
\label[appendix]{app:localization_averaging}

We record the elementary estimates used in the proof of the main theorem.
They are independent of the normalization
\(n\sqrt n\).  The estimates for the scaled remainder terms are given in
\cref{app:ll_technical_estimates}.

\subsection{Consistency and localization}

For \(\rho>0\), put
\begin{equation}
    \Xi_{n,\rho}
    \coloneqq
    \left\{
        \sup_{t\le1}
        \bigl(
            |X_t|\vee|\hat X_t^n|
        \bigr)
        \le \rho
    \right\}.
\end{equation}

\begin{lemma}[Consistency and localization]
\label[lemma]{le:ll_consistency}
\label[lemma]{le:ll_outer_localization}
Under the standing assumptions,
\begin{equation}\label{eq:ll_consistency_appendix}
    \sup_{t\le1}|X_t-\hat X_t^n|
    \xrightarrow{P}0,
\end{equation}
and
\begin{equation}\label{eq:ll_mesh_increment_consistency_appendix}
    \sup_{t\le1}
    |\hat X_t^n-\hat X_{\varphi_n(t)}^n|
    \xrightarrow{P}0.
\end{equation}
Moreover,
\begin{equation}
    \lim_{\rho\uparrow\infty}
    \limsup_{n\to\infty}
    \Prob*{P}{(\Xi_{n,\rho})^c}
    =
    0 .
\end{equation}
\end{lemma}

\begin{proof}
Fix \(\rho>|x_0|\), and set
\begin{equation}
    \tau_{n,\rho}
    \coloneqq
    \inf
    \left\{
        t\in[0,1]:
        |X_t|\vee|\hat X_t^n|>\rho
    \right\}
    \wedge1 .
\end{equation}
On \([0,\tau_{n,\rho}]\), both \(X\) and \(\hat X^n\), and also
\(\hat X^n_{\varphi_n(\cdot)}\), remain in the ball \(\{|x|\le\rho\}\).
Throughout this proof, \(C_\rho<\infty\) denotes a constant depending only
on \(\rho\) and the coefficients, whose value may change from line to line.

Let \(s\le\tau_{n,\rho}\), \(k\coloneqq\varphi_n(s)\), and \(r\in[k,s]\).
Since the coefficients in the scheme are bounded on \(\{|x|\le\rho\}\),
Gronwall's lemma gives
\[
    \sup_{u\in[k,r]}
    |\hat X_u^n-\hat X_k^n|
    \le
    C_\rho
    \left(
        n^{-1}
        +
        \sup_{u\in[k,r]}|B_u-B_k|
    \right).
\]
Taking \(r=s\) and then the supremum over \(s\le\tau_{n,\rho}\), we obtain
\begin{equation}\label{eq:ll_stopped_mesh_increment_appendix}
    \sup_{s\le\tau_{n,\rho}}
    |\hat X_s^n-\hat X_{\varphi_n(s)}^n|
    \le
    C_\rho
    \left(
        n^{-1}
        +
        \sup_{\substack{u,v\in[0,1]\\ |u-v|\le n^{-1}}}
        |B_u-B_v|
    \right).
\end{equation}
Since Brownian paths are uniformly continuous on \([0,1]\), this implies,
for each fixed \(\rho\),
\begin{equation}\label{eq:ll_stopped_mesh_increment_rewrite}
    \sup_{s\le\tau_{n,\rho}}
    |\hat X_s^n-\hat X_{\varphi_n(s)}^n|
    \xrightarrow{P}0 .
\end{equation}

Subtracting the scheme from the equation for \(X\) gives
\[
\begin{aligned}
    X_t^i-\hat X_t^{n,i}
    =
    \int_0^t
    \Bigl[
        &f^i(X_s)
        -
        f^i(\hat X_{\varphi_n(s)}^n)
        -
        f_j^i(\hat X_{\varphi_n(s)}^n)
        \bigl(
            \hat X_s^{n,j}
            -
            \hat X_{\varphi_n(s)}^{n,j}
        \bigr)
        \\
        &-
        \frac12
        \Sigma^{jk}
        f_{jk}^i(\hat X_{\varphi_n(s)}^n)
        \bigl(s-\varphi_n(s)\bigr)
    \Bigr]ds .
\end{aligned}
\]
By Taylor's formula on the ball \(\{|x|\le\rho\}\), the absolute value of
the expression in brackets is bounded on
\(\{s\le\tau_{n,\rho}\}\) by
\[
    C_\rho
    \left(
        |X_s-\hat X_s^n|
        +
        |\hat X_s^n-\hat X_{\varphi_n(s)}^n|^2
        +
        n^{-1}
    \right).
\]
Consequently, for \(t\le1\),
\[
\begin{aligned}
    \sup_{r\le t\wedge\tau_{n,\rho}}
    |X_r-\hat X_r^n|
    \le
    C_\rho
    \Bigg(
        &\int_0^t
        \sup_{u\le s\wedge\tau_{n,\rho}}
        |X_u-\hat X_u^n|\,ds
        \\
        &+
        \sup_{s\le\tau_{n,\rho}}
        |\hat X_s^n-\hat X_{\varphi_n(s)}^n|^2
        +
        n^{-1}
    \Bigg).
\end{aligned}
\]
By \eqref{eq:ll_stopped_mesh_increment_rewrite} and Gronwall's lemma,
\begin{equation}\label{eq:ll_stopped_consistency_rewrite}
    \sup_{t\le\tau_{n,\rho}}
    |X_t-\hat X_t^n|
    \xrightarrow{P}0 .
\end{equation}

We next prove the outer localization.  By the non-explosion obtained from
\eqref{eq:radial_growth_condition}, \(X\) is an \(\mathbb R^d\)-valued
continuous process on \([0,1]\).  Hence
\[
    \sup_{t\le1}|X_t|<\infty
    \qquad
    \text{a.s.}
\]
and therefore
\[
    \Prob*{P}{\sup_{t\le1}|X_t|>\rho}
    \longrightarrow0
    \qquad
    \text{as }\rho\uparrow\infty .
\]
On the event
\[
    \left\{
        \sup_{t\le1}|\hat X_t^n|>\rho,
        \,
        \sup_{t\le1}|X_t|\le\rho/2
    \right\},
\]
the first hitting time of the sphere \(\{|x|=\rho\}\) by \(\hat X^n\) is
at most \(\tau_{n,\rho}\).  Hence
\[
    \sup_{t\le\tau_{n,\rho}}
    |X_t-\hat X_t^n|
    \ge
    \rho/2 .
\]
Therefore
\[
\begin{aligned}
    \Prob*{P}{\sup_{t\le1}|\hat X_t^n|>\rho}
    &\le
    \Prob*{P}{\sup_{t\le1}|X_t|>\rho/2}
    \\
    &\quad
    +
    \Prob*{P}{
        \sup_{t\le\tau_{n,\rho}}
        |X_t-\hat X_t^n|
        \ge\rho/2
    } .
\end{aligned}
\]
Together with the preceding tail estimate for \(X\), taking
\(\limsup_{n\to\infty}\), using
\eqref{eq:ll_stopped_consistency_rewrite}, and then letting
\(\rho\uparrow\infty\), yields
\[
    \lim_{\rho\uparrow\infty}
    \limsup_{n\to\infty}
    \Prob*{P}{(\Xi_{n,\rho})^c}
    =
    0 .
\]

It remains to remove the stopping.  For every \(\delta>0\) and every
\(\rho>|x_0|\),
\[
\begin{aligned}
    \Prob*{P}{
        \sup_{t\le1}|X_t-\hat X_t^n|>\delta
    }
    &\le
    \Prob*{P}{
        \sup_{t\le\tau_{n,\rho}}
        |X_t-\hat X_t^n|>\delta
    }
    +
    \Prob*{P}{(\Xi_{n,\rho})^c}.
\end{aligned}
\]
Letting \(n\to\infty\) and then \(\rho\uparrow\infty\), we get
\eqref{eq:ll_consistency_appendix}.  The same argument, now using
\eqref{eq:ll_stopped_mesh_increment_rewrite}, gives
\eqref{eq:ll_mesh_increment_consistency_appendix}.
\end{proof}

\subsection{Local \texorpdfstring{\(L^1\)}{L1} convergence and deterministic averaging}

The next lemma is used to pass from uniform convergence in probability to
\(L^1(ds)\)-convergence after composition with a continuous function, under a
localization condition.

\begin{lemma}[Local \(L^1\)-convergence]
\label[lemma]{le:C1toconpro}
Let \(T>0\), and let \(Y^n\) and \(Y\) be measurable
\(\mathbb R^d\)-valued processes on \([0,T]\).  Let
\(h:\mathbb R^d\to\mathbb R\) be continuous.  Assume that
\[
    \sup_{s\le T}|Y_s^n-Y_s|
    \xrightarrow{P}0
\]
and
\[
    \lim_{\rho\uparrow\infty}
    \limsup_{n\to\infty}
    \Prob*{P}{
        \sup_{s\le T}
        (|Y_s^n|\vee |Y_s|)>\rho
    }
    =
    0 .
\]
Then
\begin{equation}\label{eq:local_L1_conv}
    \int_0^T |h(Y_s^n)-h(Y_s)|\,ds
    \xrightarrow{P}0 .
\end{equation}
Consequently,
\[
    \sup_{t\le T}
    \left|
        \int_0^t h(Y_s^n)\,ds
        -
        \int_0^t h(Y_s)\,ds
    \right|
    \xrightarrow{P}0 .
\]
\end{lemma}

\begin{proof}
Fix \(\varepsilon,\eta>0\).  Choose \(\rho>0\) such that
\[
    \limsup_{n\to\infty}
    \Prob*{P}{
        \sup_{s\le T}
        (|Y_s^n|\vee |Y_s|)>\rho
    }
    \le
    \eta .
\]
By uniform continuity of \(h\) on \(\{|x|\le\rho\}\), on the complementary
event,
\[
\begin{aligned}
    \int_0^T |h(Y_s^n)-h(Y_s)|\,ds
    &\le
    T
    \sup
    \Bigl\{
        |h(x)-h(y)|:
        |x|\vee|y|\le\rho,
        \\
    &\hspace{4.2cm}
        |x-y|
        \le
        \sup_{u\le T}|Y_u^n-Y_u|
    \Bigr\}.
\end{aligned}
\]
The right-hand side converges to \(0\) in probability.  Hence
\[
    \limsup_{n\to\infty}
    \Prob*{P}{
        \int_0^T |h(Y_s^n)-h(Y_s)|\,ds>\varepsilon
    }
    \le
    \eta .
\]
Since \(\eta\) is arbitrary, \eqref{eq:local_L1_conv} follows.  For the last
assertion, the left-hand side is bounded by the integral in
\eqref{eq:local_L1_conv}.
\end{proof}

The following mesh-point version is used for coefficients evaluated at the left
endpoints of the partition.

\begin{corollary}[Mesh-point \(L^1\)-convergence]
\label[corollary]{cor:local_L1_conv_mesh}
Let \(T>0\), and put
\[
    \varphi_n(s)
    \coloneqq
    \frac{\lfloor ns\rfloor}{n},
    \qquad
    0\le s\le T .
\]
Let \(Y^n\) and \(Y\) be measurable \(\mathbb R^d\)-valued processes on
\([0,T]\), and let \(h:\mathbb R^d\to\mathbb R\) be continuous.  Assume that
\[
    \sup_{s\le T}|Y_s^n-Y_s|
    \xrightarrow{P}0,
    \qquad
    \sup_{s\le T}|Y_s^n-Y_{\varphi_n(s)}^n|
    \xrightarrow{P}0,
\]
and
\[
    \lim_{\rho\uparrow\infty}
    \limsup_{n\to\infty}
    \Prob*{P}{
        \sup_{s\le T}
        (|Y_s^n|\vee |Y_s|)>\rho
    }
    =
    0 .
\]
Then
\[
    \int_0^T
    |h(Y_{\varphi_n(s)}^n)-h(Y_s)|\,ds
    \xrightarrow{P}0 .
\]
Consequently,
\[
    \sup_{t\le T}
    \left|
        \int_0^t h(Y_{\varphi_n(s)}^n)\,ds
        -
        \int_0^t h(Y_s)\,ds
    \right|
    \xrightarrow{P}0 .
\]
\end{corollary}

\begin{proof}
The assumptions imply
\[
    \sup_{s\le T}
    |Y_{\varphi_n(s)}^n-Y_s|
    \xrightarrow{P}0 .
\]
Moreover,
\[
    \sup_{s\le T}|Y_{\varphi_n(s)}^n|
    \le
    \sup_{s\le T}|Y_s^n|.
\]
Thus \cref{le:C1toconpro}, applied to the pair
\((Y_{\varphi_n(\cdot)}^n,Y)\), gives the assertion.
\end{proof}

We also use the following deterministic averaging lemma for periodic mesh
weights.

\begin{lemma}[Deterministic averaging]
\label[lemma]{le:contitoucp}
Let \(T>0\), and let \(w:[0,1]\to\mathbb R\) be bounded and measurable.  Let
\(G^n\) be real-valued measurable processes on \([0,T]\), and let \(G\) be
a continuous process on \([0,T]\).  If
\[
    \int_0^T |G_s^n-G_s|\,ds
    \xrightarrow{P}0,
\]
then
\begin{equation}\label{eq:periodic_averaging_L1}
    \sup_{t\le T}
    \left|
        \int_0^t
        G_s^n w(ns-\lfloor ns\rfloor)\,ds
        -
        \left(\int_0^1 w(u)\,du\right)
        \int_0^t G_s\,ds
    \right|
    \xrightarrow{P}0 .
\end{equation}
\end{lemma}

\begin{proof}
We have
\[
\begin{aligned}
    &\sup_{t\le T}
    \left|
        \int_0^t
        G_s^n w(ns-\lfloor ns\rfloor)\,ds
        -
        \left(\int_0^1 w(u)\,du\right)
        \int_0^t G_s\,ds
    \right|
    \\
    &\le
    \|w\|_\infty
    \int_0^T |G_s^n-G_s|\,ds
    \\
    &\quad
    +
    \sup_{t\le T}
    \left|
        \int_0^t
        G_s
        \left(
            w(ns-\lfloor ns\rfloor)
            -
            \int_0^1w(u)\,du
        \right)ds
    \right|.
\end{aligned}
\]
The first term converges to zero in probability.

It remains to treat the second term.  Let \(g\in C([0,T])\).  For
\(t\in[0,T]\), write \(m\coloneqq\lfloor nt\rfloor\).  Since the mean over
\([0,1]\) of
\[
    u
    \longmapsto
    w(u)-\int_0^1w(v)\,dv
\]
is zero,
\[
\begin{aligned}
    &\left|
        \int_0^t
        g(s)
        \left(
            w(ns-\lfloor ns\rfloor)
            -
            \int_0^1w(u)\,du
        \right)ds
    \right|
    \\
    &\le
    \sum_{k=0}^{m-1}
    \int_{k/n}^{(k+1)/n}
    |g(s)-g(k/n)|
    \left|
        w(ns-k)-\int_0^1w(u)\,du
    \right|ds
    \\
    &\quad
    +
    \int_{m/n}^{t}
    |g(s)|
    \left|
        w(ns-m)-\int_0^1w(u)\,du
    \right|ds .
\end{aligned}
\]
Thus
\[
\begin{aligned}
    &\sup_{t\le T}
    \left|
        \int_0^t
        g(s)
        \left(
            w(ns-\lfloor ns\rfloor)
            -
            \int_0^1w(u)\,du
        \right)ds
    \right|
    \\
    &\le
    2T\|w\|_\infty
    \sup_{\substack{s,r\in[0,T]\\ |s-r|\le n^{-1}}}
    |g(s)-g(r)|
    +
    2n^{-1}\|g\|_\infty\|w\|_\infty ,
\end{aligned}
\]
which tends to zero.  Since \(G\) has continuous sample paths, the preceding
estimate applied pathwise with \(g=G(\omega)\) gives
\[
    \sup_{t\le T}
    \left|
        \int_0^t
        G_s
        \left(
            w(ns-\lfloor ns\rfloor)
            -
            \int_0^1w(u)\,du
        \right)ds
    \right|
    \longrightarrow0
    \qquad
    \text{a.s.}
\]
Combining the two estimates proves \eqref{eq:periodic_averaging_L1}.
\end{proof}

\section[Technical Estimates for the Local Linearization Scheme]
{Technical Estimates for the Local\protect\\ Linearization Scheme}
\label[appendix]{app:ll_technical_estimates}

Throughout this appendix we assume
\[
    f\in C^3(\mathbb R^d;\mathbb R^d).
\]
Put
\[
    t_m^n\coloneqq m/n,
    \qquad
    m=0,\dots,n,
\]
and
\[
    I_m^n\coloneqq(t_{m-1}^n,t_m^n],
    \qquad
    m=1,\dots,n .
\]
For \(s\in I_m^n\), identities involving \(\varphi_n(s)=t_{m-1}^n\) are
understood \(ds\)-a.e.  We write
\[
    \eta_s^n\coloneqq\sigma(B_s-B_{\varphi_n(s)}).
\]
For \(\rho>0\), let
\begin{equation}
    \tau_{n,\rho}
    \coloneqq
    \inf
    \left\{
        t\in[0,1]:
        |X_t|\vee|\hat X_t^n|>\rho
    \right\}
    \wedge1 .
\end{equation}
Constants denoted by \(C_\rho\) may change from line to line, but are
independent of \(n\) and of the mesh index.  On \([0,\tau_{n,\rho}]\), all
coefficients appearing below are bounded by such constants.

We use the following localization principle throughout the appendix: if, for
every fixed \(\rho>0\),
\[
    Y^n_{\cdot\wedge\tau_{n,\rho}}
    \xrightarrow{\mathrm{ucp}}0 ,
\]
then \(Y^n\xrightarrow{\mathrm{ucp}}0\).  This follows from
\cref{le:ll_outer_localization}.

\subsection{Stopped mesh and error estimates}

\begin{lemma}[Stopped estimates]
\label[lemma]{le:ll_stopped_int_l2}
For every \(\rho>0\), there exists \(C_\rho<\infty\) such that, for
\(p=2,4\),
\begin{equation}\label{eq:ll_mesh_sup_moment_new}
    \Exp*{}{
        \sup_{s\in I_m^n}
        \left|
            \hat X^n_{s\wedge\tau_{n,\rho}}
            -
            \hat X^n_{t_{m-1}^n\wedge\tau_{n,\rho}}
        \right|^p
    }
    \le
    C_\rho n^{-p/2}.
\end{equation}
Moreover,
\begin{equation}\label{eq:ll_mesh_int_l2_new}
    \Exp*{}{
        \int_0^{\tau_{n,\rho}}
        |\hat X_s^n-\hat X_{\varphi_n(s)}^n|^2\,ds
    }
    \le
    C_\rho n^{-1},
\end{equation}
\begin{equation}\label{eq:ll_mesh_int_l2_square_new}
    \Exp*{}{
        \left(
            \int_0^{\tau_{n,\rho}}
            |\hat X_s^n-\hat X_{\varphi_n(s)}^n|^2\,ds
        \right)^2
    }
    \le
    C_\rho n^{-2},
\end{equation}
and
\begin{equation}\label{eq:ll_error_sup_l2_stopped_new}
    \Exp*{}{
        \sup_{t\le1}
        |X_{t\wedge\tau_{n,\rho}}-
        \hat X^n_{t\wedge\tau_{n,\rho}}|^2
    }
    \le
    C_\rho n^{-3} .
\end{equation}
Consequently,
\begin{equation}\label{eq:ll_error_int_l2_stopped_new}
    \Exp*{}{
        \int_0^{\tau_{n,\rho}}|X_s-\hat X_s^n|^2\,ds
    }
    \le
    C_\rho n^{-3},
\end{equation}
and, for \(m=0,\dots,n\),
\begin{equation}\label{eq:ll_error_grid_l2_stopped_new}
    \Exp*{}{
        |X_{t_m^n\wedge\tau_{n,\rho}}
        -\hat X^n_{t_m^n\wedge\tau_{n,\rho}}|^2
    }
    \le
    C_\rho n^{-3} .
\end{equation}
\end{lemma}

\begin{proof}
Fix \(m\).  On \(\{\tau_{n,\rho}>t_{m-1}^n\}\), for
\(s\in I_m^n\),
\[
\begin{aligned}
    &\hat X^n_{s\wedge\tau_{n,\rho}}
    -
    \hat X^n_{t_{m-1}^n\wedge\tau_{n,\rho}}
    =
    \int_{t_{m-1}^n}^{s\wedge\tau_{n,\rho}}
    b_u\,du
    +
    \sigma
    \bigl(
        B_{s\wedge\tau_{n,\rho}}
        -
        B_{t_{m-1}^n}
    \bigr),
\end{aligned}
\]
where, for \(u\in I_m^n\),
\[
\begin{aligned}
    b_u^i
    \coloneqq{}&
    f^i(\hat X^n_{t_{m-1}^n})
    +
    f^i_j(\hat X^n_{t_{m-1}^n})
    \bigl(
        \hat X_u^{n,j}-\hat X_{t_{m-1}^n}^{n,j}
    \bigr)
    +
    \frac12
    \Sigma^{jk}f^i_{jk}(\hat X^n_{t_{m-1}^n})
    \bigl(u-t_{m-1}^n\bigr).
\end{aligned}
\]
For \(u\le\tau_{n,\rho}\),
\[
    |b_u|
    \le
    C_\rho
    \left(
        1+|\hat X_u^n-\hat X_{t_{m-1}^n}^n|
    \right).
\]
The Burkholder--Davis--Gundy inequality and Gronwall's lemma give
\eqref{eq:ll_mesh_sup_moment_new}.  Then
\eqref{eq:ll_mesh_int_l2_new} and
\eqref{eq:ll_mesh_int_l2_square_new} follow from
\[
    \int_0^{\tau_{n,\rho}}
    |\hat X_s^n-\hat X_{\varphi_n(s)}^n|^2\,ds
    \le
    \sum_{m=1}^n
    n^{-1}
    \sup_{s\in I_m^n}
    \left|
        \hat X^n_{s\wedge\tau_{n,\rho}}
        -
        \hat X^n_{t_{m-1}^n\wedge\tau_{n,\rho}}
    \right|^2
\]
and \eqref{eq:ll_mesh_sup_moment_new} with \(p=2,4\).

It remains to prove the sharp error bound.  Subtracting the scheme from the
equation for \(X\) gives
\begin{equation}\label{eq:ll_error_decomp_strong_new}
\begin{aligned}
    X_{t\wedge\tau_{n,\rho}}^i
    -
    \hat X_{t\wedge\tau_{n,\rho}}^{n,i}
    ={}&
    \int_0^{t\wedge\tau_{n,\rho}}
    \bigl(
        f^i(X_s)-f^i(\hat X_s^n)
    \bigr)\,ds
    +
    \Gamma_{t\wedge\tau_{n,\rho}}^{n,i},
\end{aligned}
\end{equation}
where
\begin{align}
    \Gamma_t^{n,i}
    &\coloneqq
    \int_0^t
    \Bigl[
        f^i(\hat X_s^n)
        -
        f^i(\hat X_{\varphi_n(s)}^n)
        -
        f^i_j(\hat X_{\varphi_n(s)}^n)
        \bigl(
            \hat X_s^{n,j}-\hat X_{\varphi_n(s)}^{n,j}
        \bigr)
        \notag\\
    &\hspace{4.2cm}
        -
        \frac12
        \Sigma^{jk}
        f^i_{jk}(\hat X_{\varphi_n(s)}^n)
        \bigl(s-\varphi_n(s)\bigr)
    \Bigr] ds .
\end{align}
We prove
\begin{equation}\label{eq:Gamma_ll_sharp_bound_new}
    \Exp*{}{
        \sup_{t\le1}
        |\Gamma_{t\wedge\tau_{n,\rho}}^n|^2
    }
    \le
    C_\rho n^{-3} .
\end{equation}

For a.e. \(s\in[0,1]\), on \(\{s\le\tau_{n,\rho}\}\),
\begin{equation}\label{eq:theta_interval_decomp_in_stopped_proof}
\begin{aligned}
    \theta_s^{n,i}
    ={}&
    f^i(\hat X_{\varphi_n(s)}^n)
    \bigl(s-\varphi_n(s)\bigr)
    +
    \int_{\varphi_n(s)}^s
    f^i_j(\hat X_{\varphi_n(s)}^n)
    \bigl(
        \hat X_u^{n,j}-\hat X_{\varphi_n(s)}^{n,j}
    \bigr)du
    \\
    &+
    \frac14
    \Sigma^{jk}f^i_{jk}(\hat X_{\varphi_n(s)}^n)
    \bigl(s-\varphi_n(s)\bigr)^2 .
\end{aligned}
\end{equation}
Hence
\begin{equation}\label{eq:theta_l2_l4_bounds_new}
    \Exp*{}{
        1_{\{s\le\tau_{n,\rho}\}}|\theta_s^n|^2
    }
    \le
    C_\rho
    \bigl(s-\varphi_n(s)\bigr)^2,
    \qquad
    \Exp*{}{
        1_{\{s\le\tau_{n,\rho}\}}|\theta_s^n|^4
    }
    \le
    C_\rho n^{-4} .
\end{equation}
Indeed, the first and third terms in
\eqref{eq:theta_interval_decomp_in_stopped_proof} are bounded by
\(C_\rho(s-\varphi_n(s))\) and
\(C_\rho(s-\varphi_n(s))^2\), respectively, while the middle term is
estimated by the Cauchy--Schwarz inequality and
\eqref{eq:ll_mesh_sup_moment_new}.

Taylor's formula at \(\hat X_{\varphi_n(s)}^n\) gives, on
\(\{s\le\tau_{n,\rho}\}\),
\begin{align}
    &f^i(\hat X_s^n)
    -
    f^i(\hat X_{\varphi_n(s)}^n)
    -
    f^i_j(\hat X_{\varphi_n(s)}^n)
    \bigl(
        \hat X_s^{n,j}-\hat X_{\varphi_n(s)}^{n,j}
    \bigr)
    -
    \frac12
    \Sigma^{jk}f^i_{jk}(\hat X_{\varphi_n(s)}^n)
    \bigl(s-\varphi_n(s)\bigr)
    \notag\\
    &\quad =
    \frac12
    f^i_{jk}(\hat X_{\varphi_n(s)}^n)
    \left(
        \eta_s^{n,j}\eta_s^{n,k}
        -
        \Sigma^{jk}
        \bigl(s-\varphi_n(s)\bigr)
    \right)
    +
    \Lambda_s^{n,i},
    \label{eq:Gamma_taylor_decomp_new}
\end{align}
where
\[
\begin{aligned}
    \Lambda_s^{n,i}
    \coloneqq{}&
    \frac12
    f^i_{jk}(\hat X_{\varphi_n(s)}^n)
    \Bigl[
        \bigl(
            \hat X_s^{n,j}-\hat X_{\varphi_n(s)}^{n,j}
        \bigr)
        \bigl(
            \hat X_s^{n,k}-\hat X_{\varphi_n(s)}^{n,k}
        \bigr)
        -
        \eta_s^{n,j}\eta_s^{n,k}
    \Bigr]
    \\
    &+
    \int_0^1(1-\lambda)
    \Bigl[
        f^i_{jk}
        \bigl(
            \hat X_{\varphi_n(s)}^n
            +
            \lambda
            \bigl(
                \hat X_s^n-\hat X_{\varphi_n(s)}^n
            \bigr)
        \bigr)
        -
        f^i_{jk}(\hat X_{\varphi_n(s)}^n)
    \Bigr]d\lambda
    \\
    &\times
    \bigl(
        \hat X_s^{n,j}-\hat X_{\varphi_n(s)}^{n,j}
    \bigr)
    \bigl(
        \hat X_s^{n,k}-\hat X_{\varphi_n(s)}^{n,k}
    \bigr).
\end{aligned}
\]
Thus
\[
    |\Lambda_s^n|
    \le
    C_\rho
    \left(
        |\eta_s^n||\theta_s^n|
        +
        |\theta_s^n|^2
        +
        |\eta_s^n+\theta_s^n|^3
    \right).
\]
Using \eqref{eq:theta_l2_l4_bounds_new} and Brownian scaling,
\[
    \left\|
        \int_0^{\tau_{n,\rho}}
        |\eta_s^n||\theta_s^n|\,ds
    \right\|_{L^2}
    \le
    C_\rho n^{-3/2},
\]
\[
    \left\|
        \int_0^{\tau_{n,\rho}}
        |\theta_s^n|^2\,ds
    \right\|_{L^2}
    \le
    C_\rho n^{-2},
\]
and, since \(|\theta_s^n|\le C_\rho n^{-1}\) on
\(\{s\le\tau_{n,\rho}\}\),
\[
    \left\|
        \int_0^{\tau_{n,\rho}}
        |\eta_s^n+\theta_s^n|^3\,ds
    \right\|_{L^2}
    \le
    C_\rho n^{-3/2}.
\]
Therefore
\begin{equation}\label{eq:Lambda_integral_bound_new}
    \Exp*{}{
        \sup_{t\le1}
        \left|
            \int_0^{t\wedge\tau_{n,\rho}}\Lambda_s^n\,ds
        \right|^2
    }
    \le
    C_\rho n^{-3} .
\end{equation}

Let \(\chi_\rho\in C_c^\infty(\mathbb R^d)\) be equal to \(1\) on
\(\{|x|\le\rho\}\).  Define
\[
\begin{aligned}
    \Pi_t^{n,i,\rho}
    \coloneqq{}&
    \frac12
    \int_0^t
    \chi_\rho(\hat X_{\varphi_n(s)}^n)
    f^i_{jk}(\hat X_{\varphi_n(s)}^n)
    \left(
        \eta_s^{n,j}\eta_s^{n,k}
        -
        \Sigma^{jk}\bigl(s-\varphi_n(s)\bigr)
    \right)ds .
\end{aligned}
\]
The sequence \((\Pi_{t_m^n}^{n,i,\rho})_{m=0}^n\) is a discrete-time
martingale.  Brownian scaling gives, uniformly in \(m\),
\[
    \Exp*{}{
        \left|
            \Pi_{t_m^n}^{n,i,\rho}
            -
            \Pi_{t_{m-1}^n}^{n,i,\rho}
        \right|^2
        \Bigm|
        \mathcal F_{t_{m-1}^n}
    }
    \le
    C_\rho n^{-4},
\]
and
\[
    \Exp*{}{
        \sup_{t\in I_m^n}
        \left|
            \Pi_t^{n,i,\rho}
            -
            \Pi_{t_{m-1}^n}^{n,i,\rho}
        \right|^2
        \Bigm|
        \mathcal F_{t_{m-1}^n}
    }
    \le
    C_\rho n^{-4}.
\]
The discrete-time Doob inequality and summation over \(m\) yield
\begin{equation}\label{eq:Pi_sup_bound_new}
    \Exp*{}{
        \sup_{t\le1}|\Pi_t^{n,i,\rho}|^2
    }
    \le
    C_\rho n^{-3}.
\end{equation}
Since \(\chi_\rho(\hat X_{\varphi_n(s)}^n)=1\) on
\(\{s\le\tau_{n,\rho}\}\) for a.e. \(s\), the centered quadratic
term in \eqref{eq:Gamma_taylor_decomp_new},
stopped at \(\tau_{n,\rho}\), is
\(\Pi_{t\wedge\tau_{n,\rho}}^{n,i,\rho}\).  Combining
\eqref{eq:Lambda_integral_bound_new} and \eqref{eq:Pi_sup_bound_new} gives
\eqref{eq:Gamma_ll_sharp_bound_new}.

On \([0,\tau_{n,\rho}]\), \(f\) is Lipschitz.  From
\eqref{eq:ll_error_decomp_strong_new},
\[
    \sup_{r\le t}
    |X_{r\wedge\tau_{n,\rho}}-\hat X^n_{r\wedge\tau_{n,\rho}}|
    \le
    C_\rho
    \int_0^t
    \sup_{u\le s}
    |X_{u\wedge\tau_{n,\rho}}-\hat X^n_{u\wedge\tau_{n,\rho}}|\,ds
    +
    \sup_{r\le t}|\Gamma^n_{r\wedge\tau_{n,\rho}}|.
\]
Gronwall's lemma gives
\[
    \sup_{t\le1}
    |X_{t\wedge\tau_{n,\rho}}-\hat X^n_{t\wedge\tau_{n,\rho}}|
    \le
    C_\rho
    \sup_{t\le1}|\Gamma^n_{t\wedge\tau_{n,\rho}}| .
\]
Together with \eqref{eq:Gamma_ll_sharp_bound_new}, this proves
\eqref{eq:ll_error_sup_l2_stopped_new}.  The estimates
\eqref{eq:ll_error_int_l2_stopped_new} and
\eqref{eq:ll_error_grid_l2_stopped_new} follow by integration in time and by
evaluation at the grid points, respectively.
\end{proof}

\subsection{Negligibility of \texorpdfstring{\(R^{(1),n}\), \(R^{(2),n}\), and \(R^{(3),n}\)}{R1, R2, and R3}}

\begin{proposition}
\label[proposition]{prop:ll_R123_ucp_zero}
For \(\ell=1,2,3\),
\begin{equation}\label{eq:ll_R123_ucp_zero}
    R^{(\ell),n}
    \xrightarrow{\mathrm{ucp}}0 .
\end{equation}
\end{proposition}

\begin{proof}
We prove the assertion componentwise.

We first treat \(R^{(1),n}\).  Fix \(\rho>0\).  It suffices to prove
\[
    R_{\cdot\wedge\tau_{n,\rho}}^{(1),n,i}
    \xrightarrow{\mathrm{ucp}}0 .
\]
For a.e. \(s\in I_m^n\cap[0,\tau_{n,\rho}]\),
\begin{equation}\label{eq:ll_R1_G_bound_tech}
    \sup_{u\in I_m^n\cap[0,\tau_{n,\rho}]}
    \left|
        X_u-
        \hat X_{t_{m-1}^n}^n
        -
        \eta_u^n
    \right|
    \le
    \left|
        X_{t_{m-1}^n\wedge\tau_{n,\rho}}
        -
        \hat X^n_{t_{m-1}^n\wedge\tau_{n,\rho}}
    \right|
    +
    C_\rho n^{-1} .
\end{equation}
Taylor's formula applied to \(f^i_{jk}\) in
\eqref{eq:ll_taylor_remainder_def} gives, on \(\{s\le\tau_{n,\rho}\}\),
\begin{equation}
\begin{aligned}
    \widetilde R_s^{n,i}
    ={}&
    \frac16
    f^i_{jk\ell}(\hat X_{\varphi_n(s)}^n)
    \bigl(X_s^j-\hat X_{\varphi_n(s)}^{n,j}\bigr)
    \bigl(X_s^k-\hat X_{\varphi_n(s)}^{n,k}\bigr)
    \bigl(X_s^\ell-\hat X_{\varphi_n(s)}^{n,\ell}\bigr)
    +
    q_s^{n,i},
\end{aligned}
\end{equation}
where
\[
\begin{aligned}
    q_s^{n,i}
    \coloneqq{}&
    \widetilde R_s^{n,i}
    -
    \frac16
    f^i_{jk\ell}(\hat X_{\varphi_n(s)}^n)
    \bigl(X_s^j-\hat X_{\varphi_n(s)}^{n,j}\bigr)
    \bigl(X_s^k-\hat X_{\varphi_n(s)}^{n,k}\bigr)
    \bigl(X_s^\ell-\hat X_{\varphi_n(s)}^{n,\ell}\bigr).
\end{aligned}
\]
Moreover,
\begin{equation}\label{eq:ll_R1_cubic_remainder_bound_tech}
\begin{aligned}
    |q_s^{n,i}|
    \le{}&
    C_\rho
    \sup_{\substack{|z|\vee|z'|\le\rho\\
    |z-z'|\le |X_s-\hat X_{\varphi_n(s)}^n|}}
    \max_{1\le j,k,\ell\le d}
    \left|
        f^i_{jk\ell}(z)-f^i_{jk\ell}(z')
    \right|
    |X_s-\hat X_{\varphi_n(s)}^n|^3 .
\end{aligned}
\end{equation}

Using
\[
    X_s-\hat X_{\varphi_n(s)}^n
    =
    \eta_s^n
    +
    \bigl(
        X_s-\hat X_{\varphi_n(s)}^n-\eta_s^n
    \bigr),
\]
write
\begin{equation}\label{eq:ll_R1_MHK_decomp_tech}
    R_{t\wedge\tau_{n,\rho}}^{(1),n,i}
    =
    M_t^{n,i}+H_t^{n,i}+K_t^{n,i},
\end{equation}
where
\[
    M_t^{n,i}
    \coloneqq
    \frac{n\sqrt n}{6}
    \int_0^{t\wedge\tau_{n,\rho}}
    f^i_{jk\ell}(\hat X_{\varphi_n(s)}^n)
    \eta_s^{n,j}\eta_s^{n,k}\eta_s^{n,\ell}\,ds,
\]
\[
\begin{aligned}
    H_t^{n,i}
    \coloneqq{}&
    \frac{n\sqrt n}{6}
    \int_0^{t\wedge\tau_{n,\rho}}
    f^i_{jk\ell}(\hat X_{\varphi_n(s)}^n)
    \\
    &\quad\times
    \Bigl[
        \bigl(X_s^j-\hat X_{\varphi_n(s)}^{n,j}\bigr)
        \bigl(X_s^k-\hat X_{\varphi_n(s)}^{n,k}\bigr)
        \bigl(X_s^\ell-\hat X_{\varphi_n(s)}^{n,\ell}\bigr)
        -
        \eta_s^{n,j}\eta_s^{n,k}\eta_s^{n,\ell}
    \Bigr]ds,
\end{aligned}
\]
and
\[
    K_t^{n,i}
    \coloneqq
    n\sqrt n
    \int_0^{t\wedge\tau_{n,\rho}}
    q_s^{n,i}\,ds .
\]

By \cref{le:ll_stopped_int_l2},
\[
    \Exp*{}{
        \left|
            X_{t_{m-1}^n\wedge\tau_{n,\rho}}
            -
            \hat X^n_{t_{m-1}^n\wedge\tau_{n,\rho}}
        \right|^2
    }
    \le
    C_\rho n^{-3},
\]
\[
    \Exp*{}{
        \left|
            X_{t_{m-1}^n\wedge\tau_{n,\rho}}
            -
            \hat X^n_{t_{m-1}^n\wedge\tau_{n,\rho}}
        \right|
    }
    \le
    C_\rho n^{-3/2},
\]
and, since the stopped error is bounded by \(2\rho\),
\[
    \Exp*{}{
        \left|
            X_{t_{m-1}^n\wedge\tau_{n,\rho}}
            -
            \hat X^n_{t_{m-1}^n\wedge\tau_{n,\rho}}
        \right|^3
    }
    \le
    C_\rho n^{-3} .
\]
Together with \eqref{eq:ll_R1_G_bound_tech} and Brownian scaling, this implies
\[
\begin{aligned}
    \Exp*{}{\sup_{t\le1}|H_t^{n,i}|}
    &\le
    C_\rho n\sqrt n
    \sum_{m=1}^n
    \int_{I_m^n}
    \Exp*{}{
        |\eta_s^n|^2
        \left(
            \left|
                X_{t_{m-1}^n\wedge\tau_{n,\rho}}
                -
                \hat X^n_{t_{m-1}^n\wedge\tau_{n,\rho}}
            \right|
            +n^{-1}
        \right)
    }
    ds
    \\
    &\quad+
    C_\rho n\sqrt n
    \sum_{m=1}^n
    \int_{I_m^n}
    \Exp*{}{
        |\eta_s^n|
        \left(
            \left|
                X_{t_{m-1}^n\wedge\tau_{n,\rho}}
                -
                \hat X^n_{t_{m-1}^n\wedge\tau_{n,\rho}}
            \right|
            +n^{-1}
        \right)^2
    }
    ds
    \\
    &\quad+
    C_\rho n\sqrt n
    \sum_{m=1}^n
    \int_{I_m^n}
    \Exp*{}{
        \left(
            \left|
                X_{t_{m-1}^n\wedge\tau_{n,\rho}}
                -
                \hat X^n_{t_{m-1}^n\wedge\tau_{n,\rho}}
            \right|
            +n^{-1}
        \right)^3
    }
    ds
    \\
    &\le
    C_\rho n\sqrt n\cdot n
    \left(
        n^{-3}+n^{-7/2}+n^{-4}
    \right)
    \le
    C_\rho n^{-1/2}.
\end{aligned}
\]
Thus
\begin{equation}\label{eq:ll_R1_H_vanish_tech}
    H^{n,i}\xrightarrow{\mathrm{ucp}}0 .
\end{equation}

By \cref{le:ll_stopped_int_l2},
\[
    \sup_{s\le\tau_{n,\rho}}
    |X_s-\hat X_{\varphi_n(s)}^n|
    \xrightarrow{P}0 .
\]
Moreover, for \(s\in I_m^n\cap[0,\tau_{n,\rho}]\),
\eqref{eq:ll_R1_G_bound_tech} gives
\[
    |X_s-\hat X_{\varphi_n(s)}^n|^3
    \le
    C\left(
        |\eta_s^n|^3
        +
        \left|
            X_{t_{m-1}^n\wedge\tau_{n,\rho}}
            -
            \hat X^n_{t_{m-1}^n\wedge\tau_{n,\rho}}
        \right|^3
        +
        n^{-3}
    \right).
\]
Integrating over the mesh intervals and using Brownian scaling together with
the third-moment estimate for the stopped grid error yields
\[
    \sup_{n\ge1}
    \Exp*{}{
        n\sqrt n
        \int_0^{\tau_{n,\rho}}
        |X_s-\hat X_{\varphi_n(s)}^n|^3\,ds
    }
    <\infty.
\]
Consequently,
\[
    n\sqrt n
    \int_0^{\tau_{n,\rho}}
    |X_s-\hat X_{\varphi_n(s)}^n|^3\,ds
    =
    O_P(1).
\]
Let \(\omega_\rho\) be a modulus of continuity of
\((f^i_{jk\ell})_{j,k,\ell}\) on \(\{|x|\le\rho\}\).  By
\eqref{eq:ll_R1_cubic_remainder_bound_tech},
\[
\begin{aligned}
    \sup_{t\le1}|K_t^{n,i}|
    \le{}&
    C_\rho
    \omega_\rho
    \left(
        \sup_{s\le\tau_{n,\rho}}
        |X_s-\hat X_{\varphi_n(s)}^n|
    \right)
    \\
    &\quad\times
    n\sqrt n
    \int_0^{\tau_{n,\rho}}
    |X_s-\hat X_{\varphi_n(s)}^n|^3\,ds .
\end{aligned}
\]
By \cref{le:ll_consistency},
\[
    \sup_{s\le\tau_{n,\rho}}
    |X_s-\hat X_{\varphi_n(s)}^n|
    \xrightarrow{P}0.
\]
Since \(\omega_\rho(r)\to0\) as \(r\downarrow0\), it follows that
\[
    \omega_\rho
    \left(
        \sup_{s\le\tau_{n,\rho}}
        |X_s-\hat X_{\varphi_n(s)}^n|
    \right)
    \xrightarrow{P}0.
\]
Together with
\[
    n\sqrt n
    \int_0^{\tau_{n,\rho}}
    |X_s-\hat X_{\varphi_n(s)}^n|^3\,ds
    =
    O_P(1),
\]
this implies
\begin{equation}\label{eq:ll_R1_K_vanish_tech}
    K^{n,i}\xrightarrow{\mathrm{ucp}}0.
\end{equation}

It remains to treat \(M^{n,i}\).  Let
\(\psi_\rho\in C_c^\infty(\mathbb R^d)\) be equal to \(1\) on
\(\{|x|\le\rho\}\).  Put
\[
    \overline M_t^{n,i}
    \coloneqq
    \frac{n\sqrt n}{6}
    \int_0^t
    \psi_\rho(\hat X_{\varphi_n(s)}^n)
    f^i_{jk\ell}(\hat X_{\varphi_n(s)}^n)
    \eta_s^{n,j}\eta_s^{n,k}\eta_s^{n,\ell}\,ds .
\]
Then \(M_t^{n,i}=\overline M_{t\wedge\tau_{n,\rho}}^{n,i}\).  The sequence
\((\overline M_{t_m^n}^{n,i})_{m=0}^n\) is a discrete-time martingale, because
the conditional third moments of the Brownian increments vanish.  Brownian
scaling gives, uniformly in \(m\),
\[
    \Exp*{}{
        \left|
            \overline M_{t_m^n}^{n,i}
            -
            \overline M_{t_{m-1}^n}^{n,i}
        \right|^2
        \Bigm|
        \mathcal F_{t_{m-1}^n}
    }
    \le
    C_\rho n^{-2},
\]
and
\[
    \Exp*{}{
        \sup_{t\in I_m^n}
        \left|
            \overline M_t^{n,i}
            -
            \overline M_{t_{m-1}^n}^{n,i}
        \right|^2
        \Bigm|
        \mathcal F_{t_{m-1}^n}
    }
    \le
    C_\rho n^{-2}.
\]
The discrete-time Doob inequality and summation over \(m\) therefore yield
\[
    \Exp*{}{
        \sup_{t\le1}|\overline M_t^{n,i}|^2
    }
    \le
    C_\rho n^{-1}.
\]
Therefore
\begin{equation}\label{eq:ll_R1_M_vanish_tech}
    M^{n,i}\xrightarrow{\mathrm{ucp}}0 .
\end{equation}
Combining
\eqref{eq:ll_R1_MHK_decomp_tech},
\eqref{eq:ll_R1_H_vanish_tech},
\eqref{eq:ll_R1_K_vanish_tech}, and
\eqref{eq:ll_R1_M_vanish_tech}, and then removing the localization, we get
\begin{equation}\label{eq:ll_R1_ucp_zero_tech}
    R^{(1),n,i}\xrightarrow{\mathrm{ucp}}0 .
\end{equation}

We next consider \(R^{(2),n}\).  On \([0,\tau_{n,\rho}]\),
\[
\begin{aligned}
    \sup_{t\le1}
    |R_{t\wedge\tau_{n,\rho}}^{(2),n,i}|
    &\le
    C_\rho n\sqrt n
    \int_0^{\tau_{n,\rho}}
    |\hat X_s^n-\hat X_{\varphi_n(s)}^n|
    |X_s-\hat X_s^n|\,ds
    \\
    &\le
    C_\rho n\sqrt n
    \left(
        \int_0^{\tau_{n,\rho}}
        |\hat X_s^n-\hat X_{\varphi_n(s)}^n|^2\,ds
    \right)^{1/2}
    \left(
        \int_0^{\tau_{n,\rho}}
        |X_s-\hat X_s^n|^2\,ds
    \right)^{1/2}.
\end{aligned}
\]
Using \eqref{eq:ll_mesh_int_l2_new} and
\eqref{eq:ll_error_int_l2_stopped_new},
\[
    \Exp*{}{
        \sup_{t\le1}
        |R_{t\wedge\tau_{n,\rho}}^{(2),n,i}|
    }
    \le
    C_\rho n^{-1/2}.
\]
Thus
\begin{equation}
    R^{(2),n,i}\xrightarrow{\mathrm{ucp}}0 .
\end{equation}

Finally,
\[
    \sup_{t\le1}
    |R_{t\wedge\tau_{n,\rho}}^{(3),n,i}|
    \le
    C_\rho n\sqrt n
    \int_0^{\tau_{n,\rho}}|X_s-\hat X_s^n|^2\,ds .
\]
By \eqref{eq:ll_error_int_l2_stopped_new},
\[
    \Exp*{}{
        \sup_{t\le1}
        |R_{t\wedge\tau_{n,\rho}}^{(3),n,i}|
    }
    \le
    C_\rho n^{-3/2}.
\]
Hence
\begin{equation}\label{eq:ll_R3_ucp_zero_tech}
    R^{(3),n,i}\xrightarrow{\mathrm{ucp}}0 .
\end{equation}
The preceding three convergences prove
\eqref{eq:ll_R123_ucp_zero}.
\end{proof}

\subsection{Negligibility of \texorpdfstring{\(R^{(4),n}\) and \(R^{(5),n}\)}{R4 and R5}}

Write
\begin{equation}\label{eq:theta_interval_decomposition}
    \theta_s^n
    =
    \theta_s^{n,(1)}+
    \theta_s^{n,(2)}+
    \theta_s^{n,(3)},
\end{equation}
where
\begin{align}
    \theta_s^{n,(1),i}
    &\coloneqq
    f^i(\hat X_{\varphi_n(s)}^n)
    \bigl(s-\varphi_n(s)\bigr),
    \\
    \theta_s^{n,(2),i}
    &\coloneqq
    \int_{\varphi_n(s)}^s
    f^i_j(\hat X_{\varphi_n(s)}^n)
    \bigl(
        \hat X_u^{n,j}-\hat X_{\varphi_n(s)}^{n,j}
    \bigr)\,du,
    \\
    \theta_s^{n,(3),i}
    &\coloneqq
    \frac14
    \Sigma^{jk}f^i_{jk}(\hat X_{\varphi_n(s)}^n)
    \bigl(s-\varphi_n(s)\bigr)^2 .
\end{align}

\begin{lemma}[Finite-variation increment estimates]
\label[lemma]{le:theta_estimates}
For every \(\rho>0\),
\begin{equation}\label{eq:theta_L2_integral_bound}
    \Exp*{}{
        \int_0^{\tau_{n,\rho}}|\theta_s^n|^2\,ds
    }
    \le
    C_\rho n^{-2} .
\end{equation}
Moreover,
\begin{equation}\label{eq:Ctheta_eta_bound}
    \Exp*{}{
        n\sqrt n
        \int_0^{\tau_{n,\rho}}
        |\theta_s^{n,(2)}|\,|\eta_s^n|\,ds
    }
    \le
    C_\rho n^{-1/2},
\end{equation}
and
\begin{equation}\label{eq:Dtheta_eta_bound}
    \Exp*{}{
        n\sqrt n
        \int_0^{\tau_{n,\rho}}
        |\theta_s^{n,(3)}|\,|\eta_s^n|\,ds
    }
    \le
    C_\rho n^{-1} .
\end{equation}
\end{lemma}

\begin{proof}
On \(\{s\le\tau_{n,\rho}\}\),
\[
    |\theta_s^{n,(1)}|
    \le
    C_\rho\bigl(s-\varphi_n(s)\bigr),
    \qquad
    |\theta_s^{n,(3)}|
    \le
    C_\rho\bigl(s-\varphi_n(s)\bigr)^2 .
\]
Moreover, by the Cauchy--Schwarz inequality and \eqref{eq:ll_mesh_sup_moment_new},
\[
    \Exp*{}{
        1_{\{s\le\tau_{n,\rho}\}}
        |\theta_s^{n,(2)}|^2
    }
    \le
    C_\rho
    \bigl(s-\varphi_n(s)\bigr)^2 n^{-1}
    \le
    C_\rho n^{-3} .
\]
Integration over the mesh intervals gives
\[
    \Exp*{}{
        \int_0^{\tau_{n,\rho}}|\theta_s^{n,(1)}|^2\,ds
    }
    \le
    C_\rho n^{-2},
    \qquad
    \Exp*{}{
        \int_0^{\tau_{n,\rho}}|\theta_s^{n,(2)}|^2\,ds
    }
    \le
    C_\rho n^{-3},
\]
and
\[
    \Exp*{}{
        \int_0^{\tau_{n,\rho}}|\theta_s^{n,(3)}|^2\,ds
    }
    \le
    C_\rho n^{-4}.
\]
This proves \eqref{eq:theta_L2_integral_bound}.

For the second estimate, Brownian scaling gives
\[
    \Exp*{}{|\eta_s^n|^2}
    \le
    C_\rho
    \bigl(s-\varphi_n(s)\bigr).
\]
Thus
\[
\begin{aligned}
    \Exp*{}{
        \int_0^{\tau_{n,\rho}}
        |\theta_s^{n,(2)}|\,|\eta_s^n|\,ds
    }
    &\le
    \sum_{m=1}^n
    \int_{I_m^n}
    \left(
        \Exp*{}{
            1_{\{s\le\tau_{n,\rho}\}}|\theta_s^{n,(2)}|^2
        }
    \right)^{1/2}
    \left(
        \Exp*{}{|\eta_s^n|^2}
    \right)^{1/2}ds
    \\
    &\le
    C_\rho
    n\cdot n^{-3/2}
    \int_0^{1/n} r^{1/2}\,dr
    \le
    C_\rho n^{-2} .
\end{aligned}
\]
Multiplying by \(n\sqrt n\) gives \eqref{eq:Ctheta_eta_bound}.

Finally,
\[
    \Exp*{}{|\eta_s^n|}
    \le
    C_\rho
    \bigl(s-\varphi_n(s)\bigr)^{1/2}.
\]
Hence
\[
    \Exp*{}{
        \int_0^{\tau_{n,\rho}}
        |\theta_s^{n,(3)}|\,|\eta_s^n|\,ds
    }
    \le
    C_\rho
    n\int_0^{1/n}r^{5/2}\,dr
    \le
    C_\rho n^{-5/2}.
\]
After multiplication by \(n\sqrt n\), this gives
\eqref{eq:Dtheta_eta_bound}.
\end{proof}

\begin{proposition}
\label[proposition]{prop:ll_R45_ucp_zero}
For \(\ell=4,5\),
\begin{equation}\label{eq:ll_R45_ucp_zero}
    R^{(\ell),n}
    \xrightarrow{\mathrm{ucp}}0 .
\end{equation}
\end{proposition}

\begin{proof}
We again work componentwise.

For \(R^{(4),n}\), on \([0,\tau_{n,\rho}]\),
\[
    \sup_{t\le1}
    |R_{t\wedge\tau_{n,\rho}}^{(4),n,i}|
    \le
    C_\rho n\sqrt n
    \int_0^{\tau_{n,\rho}}|\theta_s^n|^2\,ds .
\]
By \eqref{eq:theta_L2_integral_bound},
\[
    \Exp*{}{
        \sup_{t\le1}
        |R_{t\wedge\tau_{n,\rho}}^{(4),n,i}|
    }
    \le
    C_\rho n^{-1/2}.
\]
Therefore
\begin{equation}\label{eq:ll_R4_ucp_zero_technical}
    R^{(4),n,i}
    \xrightarrow{\mathrm{ucp}}0 .
\end{equation}

We turn to \(R^{(5),n}\).  Using
\eqref{eq:theta_interval_decomposition}, decompose the stopped process into
the three terms obtained by replacing \(\theta^n\) in \(R^{(5),n}\) by
\(\theta^{n,(1)}\), \(\theta^{n,(2)}\), and \(\theta^{n,(3)}\), respectively.
The suprema of the \(\theta^{n,(2)}\)- and \(\theta^{n,(3)}\)-contributions are
bounded by
\[
    C_\rho n\sqrt n
    \int_0^{\tau_{n,\rho}}
    |\theta_s^{n,(2)}|\,|\eta_s^n|\,ds
\]
and
\[
    C_\rho n\sqrt n
    \int_0^{\tau_{n,\rho}}
    |\theta_s^{n,(3)}|\,|\eta_s^n|\,ds,
\]
respectively.  Hence \eqref{eq:Ctheta_eta_bound},
\eqref{eq:Dtheta_eta_bound}, and Markov's inequality imply that these two
contributions converge to zero ucp.

It remains to treat the \(\theta^{n,(1)}\)-part.  Let
\(\psi_\rho\in C_c^\infty(\mathbb R^d)\) be equal to \(1\) on
\(\{|x|\le\rho\}\), and define
\[
\begin{aligned}
    J_t^{n,i,\rho}
    \coloneqq{}&
    n\sqrt n
    \int_0^t
    \psi_\rho(\hat X_{\varphi_n(s)}^n)
    f^i_{jk}(\hat X_{\varphi_n(s)}^n)
    f^j(\hat X_{\varphi_n(s)}^n)
    \bigl(s-\varphi_n(s)\bigr)
    \sigma^k_\alpha
    \bigl(
        B_s^\alpha-B_{\varphi_n(s)}^\alpha
    \bigr)ds .
\end{aligned}
\]
Since \(\psi_\rho(\hat X_{\varphi_n(s)}^n)=1\) whenever
\(s\le\tau_{n,\rho}\), the stopped \(\theta^{n,(1)}\)-part at time \(t\)
equals \(J_{t\wedge\tau_{n,\rho}}^{n,i,\rho}\).  The sequence
\[
    \bigl(
        J_{t_m^n}^{n,i,\rho}
    \bigr)_{m=0}^n
\]
is a discrete-time martingale.  For each \(m\),
Brownian scaling gives
\[
    \Exp*{}{
        \left|
            \int_{t_{m-1}^n}^{t_m^n}
            \bigl(s-t_{m-1}^n\bigr)
            \bigl(
                B_s-B_{t_{m-1}^n}
            \bigr)ds
        \right|^2
        \Bigm|
        \mathcal F_{t_{m-1}^n}
    }
    \le
    C n^{-5}.
\]
Thus the conditional second moments of the scaled mesh increments are bounded
by \(C_\rho n^{-2}\).  Brownian scaling also gives, uniformly in \(m\),
\[
    \Exp*{}{
        \sup_{t\in I_m^n}
        \left|
            J_t^{n,i,\rho}
            -
            J_{t_{m-1}^n}^{n,i,\rho}
        \right|^2
        \Bigm|
        \mathcal F_{t_{m-1}^n}
    }
    \le
    C_\rho n^{-2}.
\]
The discrete-time Doob inequality and summation over \(m\) therefore yield
\[
\begin{aligned}
    &\Exp*{}{
        \sup_{t\le1}
        \left|
            n^{3/2}
            \int_0^t
            \psi_\rho(\hat X_{\varphi_n(s)}^n)
            f^i_{jk}(\hat X_{\varphi_n(s)}^n)
            f^j(\hat X_{\varphi_n(s)}^n)
            (s-\varphi_n(s))
            \sigma^k_\alpha
            (B_s^\alpha-B_{\varphi_n(s)}^\alpha)ds
        \right|^2
    }
    \\
    &\quad\le
    C_\rho n^{-1} .
\end{aligned}
\]
Hence the \(\theta^{n,(1)}\)-part converges to zero ucp.  Therefore
\begin{equation}\label{eq:ll_R5_ucp_zero_technical}
    R^{(5),n,i}
    \xrightarrow{\mathrm{ucp}}0 .
\end{equation}

The componentwise convergences
\eqref{eq:ll_R4_ucp_zero_technical} and
\eqref{eq:ll_R5_ucp_zero_technical} prove
\eqref{eq:ll_R45_ucp_zero}.
\end{proof}

%%%%% reference %%%%%

\printbibliography % biblatex

%\bibliographystyle{jname} %参考文献
%\bibliography{probref}

\end{document}